\documentclass[letter]{amsart}

\usepackage{amssymb}
\usepackage{amsfonts}
\usepackage{amsmath}
\usepackage{graphicx}
\usepackage{mathrsfs}
\usepackage{dsfont}
\usepackage{amscd}
\usepackage{multirow}
\usepackage{mathpazo}
\usepackage[all]{xy}
\usepackage[scaled]{helvet} 
\usepackage{courier} 
\normalfont
\usepackage[T1]{fontenc}
\usepackage{verbatim}
\usepackage{hyperref}

\newcommand{\N}{{\mathds{N}}}
\newcommand{\Z}{{\mathds{Z}}}

\newcommand{\R}{{\mathds{R}}}
\newcommand{\C}{{\mathds{C}}}

\newcommand{\A}{{\mathfrak{A}}}
\newcommand{\B}{{\mathfrak{B}}}
\newcommand{\M}{{\mathfrak{M}}}
\newcommand{\Lip}{{\mathsf{L}}}

\newcommand{\cBall}[3]{{\mathscr{B}_{#1}(#2,#3)}}

\newcommand{\Kantorovich}[1]{{\mathsf{mk}_{#1}}}
\newcommand{\boundedlipschitz}[1]{{\mathsf{bl}_{#1}}}

\newcommand{\StateSpace}{{\mathscr{S}}}
\newcommand{\unital}[1]{{\mathfrak{u}{#1}}}
\newcommand{\mongekant}{{extended Mon\-ge-Kan\-to\-ro\-vich metric}}
\newcommand{\qms}{quantum locally compact metric space}
\newcommand{\qmss}{quantum locally compact separable metric space}
\newcommand{\tu}{{topographic uniform topology}}
\newcommand{\tut}{{\textbf{\emph{tu}}}}
\newcommand{\unit}{1}
\newcommand{\pnorm}[1]{{\mathsf{P}_{#1}}}
\newcommand{\sa}[1]{{\mathfrak{sa}({#1})}}
\newcommand{\mulip}{{\mathfrak{L}_1}}
\newcommand{\indicator}[1]{{\chi_{#1}}}
\newcommand{\compacts}[1]{{\mathcal{K}\left({#1}\right)}}

\newcommand{\bigslant}[2]{{\raisebox{.2em}{$#1$}\left/\raisebox{-.2em}{$#2$}\right.}}
\newcommand{\corner}[2]{{\indicator{#1}{#2}\indicator{#1}}}

\theoremstyle{plain}
\newtheorem{theorem}{Theorem}[section]

\newtheorem{condition}[theorem]{Condition}

\newtheorem{corollary}[theorem]{Corollary}

\newtheorem{lemma}[theorem]{Lemma}
\newtheorem{proposition}[theorem]{Proposition}

\theoremstyle{definition}
\newtheorem{definition}[theorem]{Definition}

\newtheorem{notation}[theorem]{Notation}

\newtheorem{convention}[theorem]{Convention}

\theoremstyle{remark}

\newtheorem{counterexample}[theorem]{Counterexample}

\newtheorem{remark}[theorem]{Remark}

\renewcommand{\geq}{\geqslant}
\renewcommand{\leq}{\leqslant}

\numberwithin{equation}{section}

\begin{document}
\title[Quantum Locally Compact Metric Spaces]{Quantum Locally Compact Metric Spaces}
\author{Fr\'{e}d\'{e}ric Latr\'{e}moli\`{e}re}
\email{frederic@math.du.edu}
\urladdr{http://www.math.du.edu/\symbol{126}frederic}
\address{Department of Mathematics \\ University of Denver \\ Denver CO 80208}
\dedicatory{Dedicated to Marc Rieffel.}

\date{\today}

\begin{abstract}
We introduce the notion of a {\qms}, which is the noncommutative analogue of a locally compact metric space, and generalize to the non-unital setting the notion of quantum metric spaces introduced by Rieffel. We then provide several examples of such structures, including the Moyal plane, compact quantum metric spaces and locally compact metric spaces. This paper provides an answer to the question raised in the literature about the proper notion of a quantum metric space in the nonunital setup and offers important insights into noncommutative geometry for non compact quantum spaces.
\end{abstract}
\maketitle


\section{Introduction}

Noncommutative metric geometry is the study of noncommutative generalizations of algebras of Lipschitz functions on metric spaces. Inspired by the work of Connes \cite{Connes89,Connes}, Rieffel introduced in \cite{Rieffel98a, Rieffel99} the notion of a compact quantum metric space and in \cite{Rieffel00} a generalization of the Gromov-Hausdorff distance \cite{Gromov}, thus providing in \cite{Latremoliere05, Rieffel01} a meaning to many approximations of classical and quantum spaces by matrix algebras found in the physics literature (see for instance \cite{Connes97,tHooft02}), and pioneering a new set of techniques in the study of the geometry of C*-algebras (a sample of which is \cite{Ozawa05,Rieffel02,Rieffel08,Rieffel10, Rieffel10b}). Our work in this paper offers an answer to the problem of finding a noncommutative analogue for Lipschitz algebras on \emph{locally compact metric spaces}, in preparation for the notion of Gromov-Hausdorff convergence for {\qms s} which we lay out in a coming paper \cite{Latremoliere12c}. Non-compact locally compact quantum spaces with candidates for metrics arise naturally in various domains. In physics, the Moyal plane \cite{Cagnache11, Varilly04, Martinetti11} has been equipped with a natural spectral triple and we prove in this paper that it gives the Moyal plane the structure of an unbounded {\qms}. Other examples of interest whose study is postponed to later papers, but which hint at the scope of the notion we introduce in this article, are the C*-algebra for space-time uncertainty relations \cite{Doplicher95,Rieffel96}, cross-products on non-compact spaces \cite{Bellissard10}, C*-algebras of foliations and many more. Thus in particular, our work provides the foundation for metric noncommutative geometry in all these contexts. 

Our work is based on the extension to the noncommutative setting of the following picture. Let $(X,m)$ be a locally compact metric space. For any function $f : X\rightarrow \R$, we define the Lipschitz constant of $f$ as:
\begin{equation*}
\mathsf{Lip}_m(f) = \sup \left\{ \frac{|f(y)-f(x)|}{m(y,x)} : x,y \in X, x\not= y \right\}\text{.}
\end{equation*}
The function $\mathsf{Lip}_m$ is a seminorm on its domain $\mathfrak{L}=\{f : X \rightarrow \R : \mathsf{Lip}_m(f) < \infty \}$, and this domain is in fact a subalgebra of the algebra of all $\R$-valued continuous functions on $X$. If, in particular, $(X,m)$ is compact, then $\mathfrak{L}$  is a a dense subalgebra of the self-adjoint part of the C*-algebra $C(X)$ of $\C$-valued continuous functions on $X$. If $(X,m)$ is locally compact but not compact, then Gel'fand duality theory suggests that the proper replacement for $C(X)$ is the C*-algebra $C_0(X)$ of continuous $\C$-valued functions on $X$, vanishing at infinity. In this case, the subset $\mathfrak{L}_0$ of $\mathfrak{L}$ consisting of Lipschitz functions vanishing at infinity, is a dense subalgebra of the self-adjoint part of $C_0(X)$.

A central observation to noncommutative metric geometry is that the distance $m$ can be recovered from its associated Lipschitz seminorm. To this end, one constructs a metric on the state space $\StateSpace(C_0(X))$, i.e. the set of all integrals against Radon probability measures on $X$. While there are many metrics on this space \cite{Billingsley,Dudley}, the one of particular interest for our purpose is the {\mongekant} introduced by Kantorovich in \cite{Kantorovich40}. Kantorovich and Rubinstein proved in \cite{Kantorovich58} that this metric can be defined as:
\begin{multline}\label{wasskant-intro}
\Kantorovich{\mathsf{Lip}_m} : \mu,\nu \in \StateSpace(C_0(X)) \longmapsto \\ \sup \{ |\mu(f)-\nu(f) | : f \in C_0(X) \text{ and } \mathsf{Lip}_m(f) \leq 1 \} \text{.}
\end{multline}

It is easy to check that the restriction of $\Kantorovich{\mathsf{Lip}_m}$ to the set $X$ identified with the set of Dirac probability measures over $X$ is indeed $m$. Moreover, $\Kantorovich{\mathsf{Lip}_m}$ has a very good topological property when $(X,m)$ is compact: the topology it induces on $\StateSpace(C(X))$ is the weak* topology.

The analogue of a compact metric space, introduced by Rieffel in \cite{Rieffel98a}, is thus a pair $(\A,\Lip)$ of an order-unit space $\A$ (i.e. a subspace of the self-adjoint part of a unital C*-algebra, containing the unit) and a densely defined seminorm $\Lip$ on $\A$ with the properties that the distance defined on the state space $\StateSpace(\A)$ of $\A$ by the analogue of Identity (\ref{wasskant-intro}), with $\mathsf{Lip}_m$ replaced with $\Lip$ and $C_0(X)$ by $\A$, gives $\StateSpace(\A)$ a finite radius and the weak* topology. This leads to a rich theory as illustrated, for example, in \cite{Rieffel00,Rieffel01,Latremoliere05}.

However, difficulties arise when trying to find a noncommutative analogue of a non-compact, locally compact metric space  \cite{Latremoliere05b}. The first and evident problem is that the state space of a nonunital C*-algebra is not a weak* locally compact space. The second matter is that the Monge-Kantorovich construction no longer gives a distance, but rather an extended metric, i.e. it takes the value $\infty$ on some pairs of probability measures. Last, even if one restricts attention to bounded subsets of the state space for the {\mongekant}, the topology induced by the metric is usually strictly stronger than the weak* topology. All these matters are attributable to one main feature of the non-compact case: points, and more generally probability measures, can escape at infinity.

In \cite{Latremoliere05b}, a first approach to this problem was chosen, where the {\mongekant} is replaced by a bounded form called the bounded-Lip\-schitz distance. A summary of the main result of our paper \cite{Latremoliere05b} for our current purpose is given by:

\begin{theorem}[\cite{Latremoliere05b}]\label{BoundedLipschitz}
Let $\A$ be a separable C*-algebra with norm $\|\cdot\|_\A$ and let $\B$ be a bounded total subset of the self-adjoint part of $\A$. For any two states $\varphi,\psi$ of $\A$ we define:
\begin{equation*}
\boundedlipschitz{\B}(\varphi,\psi) = \sup \{ |\varphi(a) - \psi(a)| : a \in \B \} \text{.}
\end{equation*}
Then $\boundedlipschitz{\B}$ is a distance on the state space $\StateSpace(\A)$ and the following are equivalent:
\begin{itemize}
\item The distance $\boundedlipschitz{\B}$ metrizes the restriction of the weak* topology $\sigma(\A^\ast,\A)$ to $\StateSpace(\A)$,
\item There exists a strictly positive $h$ of $\A$ such that the set $h\B h$ is totally bounded for $\|\cdot\|_{\A}$,
\item For any strictly positive element $h$ of $\A$, the set $h\B h$ is totally bounded for $\|\cdot\|_{\A}$,
\item For all positive $a,b \in \A$, the set $a\B b$ is totally bounded for $\|\cdot\|_{\A}$,
\item The set $\B$ is totally bounded in the \emph{weakly uniform topology} on $\A$.
\end{itemize}
\end{theorem}

This theorem was a consequence of the metrizability on bounded subsets of the weakly-uniform topology introduced in \cite{Latremoliere05b}. This theorem is particularly well-suited to define noncompact finite-diameter locally compact quantum metric spaces in the spirit of \cite{Rieffel98a}. Indeed, assume that we are given a nonunital C*-algebra $\A$ and a norm $\Lip$ defined on a dense subset of the self-adjoint part $\sa{\A}$ of $\A$, and we set $\B = \{ a \in \sa{\A} : \Lip(a) \leq 1 \}$. Then, should $\B$ be norm bounded in $\A$, the distance $\boundedlipschitz{\B}$, which is the Monge-Kantorovich metric associated with $\Lip$, gives $\StateSpace(\A)$ a finite diameter, and Theorem (\ref{BoundedLipschitz}) provides us with a criterion for the topology induced by $\boundedlipschitz{\B}$ to be the weak* topology on $\StateSpace(\A)$. 

One solution to the difficulties which arise when the set $\B = \{a \in \sa{\A} : \Lip(a) \leq 1\}$ is no longer norm bounded, such as the fact $\boundedlipschitz{\B}$ is no longer a metric nor does the topology it induces on $\StateSpace(\A)$ agree with the weak* topology, is to replace the {\mongekant} by the family of Bounded-Lipschitz metrics $(\boundedlipschitz{\B_r})_{r > 0}$ where $\B_r = \{ a \in \A : \Lip(a)\leq 1\text{ and } \|a\|_\A \leq r\}$ for all real numbers $r>0$. This was the approach we studied in \cite{Latremoliere05b}, and using Theorem (\ref{BoundedLipschitz}), one get a criterion for all these metrics to metrize the weak* topology on the whole state space.

In this paper, we address the natural question left open by the bounded-Lip\-schitz approach of \cite{Latremoliere05b}, which concerns using the {\mongekant} in the setting of general locally compact quantum metric spaces, with no restriction of diameter or substitution of metrics. The main motivation lies with the construction of a Gromov-Hausdorff convergence theory \cite{Gromov} for noncommutative geometries even in the non-compact case.  In some sense, the bounded-Lip\-schitz approach artificially restricts the geometry of the underlying quantum space and only sees the space ``locally'', i.e it does not allow to recover the full metric in the classical case: for instance, if $\A=C_0(\R)$ and $\Lip$ is the Lipschitz seminorm for the ordinary metric of $\R$, then the bounded-Lip\-schitz $\boundedlipschitz{\mathfrak{B}_1}$ endowed $\R$ with the distance $x,y \mapsto \min\{|x-y|,1\}$. Thus, the {\mongekant} seems the proper tool to consider for noncommutative metric geometry, despite its occasional ill-behavior.

The solution we offer in this paper is based on the fundamental observation of Dobrushin \cite[Theorem 2]{Dobrushin70} regarding the {\mongekant} for non-compact metric spaces. The key is that this extended metric is well-behaved when restricted to subsets of the state space with good behavior at infinity.  To frame this discussion, it is useful to recall that, by the Prohorov's theorem \cite{prohorov56,Dudley}, a subset $\mathscr{P}$ of the space of Borel probability measures over a locally compact space $X$ has the property that its weak* closure is still a set of Borel probability measures if and only if it is \emph{tight}, meaning:
\begin{equation*}
\forall\varepsilon > 0 \;\; \exists K \subseteq X\;\;\text{$K$ is compact}\text{ and } (\forall \mu \in \mathscr{P} \;\; \mu(X\setminus K) \leq \varepsilon) \text{.}
\end{equation*}
This criterion is not metric, however, and there are in general weak* compact subsets of the state space of $C_0(X)$ which are not metrized by the restriction of the {\mongekant} (see Counterexample \ref{noncompact-closed-balls-cex}). Dobrushin introduces in \cite{Dobrushin70} a stronger and natural form of tightness adapted to the metric situation. A set $\mathscr{P}$ of Borel probability measures on a locally compact metric space $(X,m)$ is \emph{Dobrushin-tight} if, for some $x_0 \in X$, we have:
\begin{equation}\label{Dobrushin-tight-intro}
\lim_{r \rightarrow \infty} \sup \left\{ \int_{\{x \in X : m(x,x_0) \geq r\}} m(x,x_0)\,d\mu(x) : \mu \in \mathscr{P} \right\} = 0\text{.}
\end{equation}
This notion does not depend on the choice of the base point $x_0$, by the triangle inequality. We note that a Dobrushin-tight set is tight when $(X,m)$ has infinite diameter, though all sets of probability measures are Dobrushin-tight on a bounded locally compact metric space. Dobrushin proves in \cite[Theorem 2]{Dobrushin70} that the {\mongekant} restricted to a Dobrushin-tight set induces the relative weak* topology on this set.

We are thus led in this paper to introduce a notion of a good behavior at infinity of a set of states of a C*-algebra for metric purposes. Unlike in the commutative case described above, however, a generic way to talk about behavior at infinity within a C*-algebra does not seem to exist --- a problem already encountered for quite a different problem in \cite{Akemann89}, for instance.

Moreover, since our main intent is to provide a framework for quantum Gro\-mov-Hausdorff distance, we are faced with the need for a notion of locality as well. Indeed, convergence for locally compact metric spaces in the sense of Gromov \cite{Gromov} requires to work in the category of pointed metric spaces, rather than just metric spaces. Convergence is then defined in terms of convergence of closed balls around the chosen based points. Thus, we have to understand how to describe a local structure of a C*-algebra which allows for a definition of behavior at infinity appropriate for a generalization of Dobrushin tightness.  Of course, the notion of local behavior is in essence what becomes ambiguous when going from the commutative to the noncommutative world. 

In the noncommutative world, our suggestion is to pick a favored set of ``observables'' for which it is possible to talk about locality --- i.e. a commutative set --- and accept that our notions of locality and behavior at infinity will depend on this choice, as those observables which do not commute with our chosen set can typically be ``spread'' all over our quantum space. Therefore, we consider a \emph{topographic quantum space} as a pair $(\A,\M)$ where $\A$ is a C*-algebra and $\M$ is an Abelian C*-subalgebra of $\A$ which contains an approximate unit of $\A$. The Abelian nature of $\M$ allows for local definitions while the existence of an approximate unit allows for the discussion of behavior at infinity. Such a structure is a at the root of our work in this paper. The terminology we chose is inspired by the picture we shall see as characteristic of the separable case, where in fact, we will pick our Abelian C*-subalgebras of the form $C^\ast(h)$ with $h$ a strictly positive element which plays the role of an ``altitude'' function, with the level sets drawing a sort of topographic map of the underlying space. 

With these ingredients, it becomes possible to formulate a definition for quantum locally compact metric spaces. Our paper progresses toward and reaches this goal as follows.

We start with some needed results about Lipschitz pairs, i.e. C*-algebras whose unitization carries on its self-adjoint part a densely defined seminorm which vanishes only on the scalar multiple of the unit. We define the {\mongekant} in this context and prove that it gives a finer topology than the weak* topology. We then discuss useful properties of topographic quantum spaces, which lay the foundations for our notion of well-behaved sets of states. When put together, Lipschitz pairs and topographic quantum spaces form Lipschitz triples, which have the same signature as {\qms s} and where we can define the concept of tame sets of states, which are our analogues of Dobrushin-tight sets, though a stronger notion even in the commutative case, as a tame set is always tight.

We are then able to introduce the notion of a {\qms}. Following Rieffel's terminology, the Lipschitz seminorm of a {\qms} is called a Lip-norm. The main purpose of the third section is to provide several characterizations of {\qms s} among Lipschitz triples by means of a topological requirement on a subset of the C*-algebra associated with the Lip-norm. The main tool is the construction of a bridge topology which is inspired by the weakly uniform topology of \cite{Latremoliere05b}, though it is typically weaker and, more importantly, it depends on the topographic structure. This bridge topology leads us to our main characterization for {\qms}. We then derive two more characterizations, one which fit naturally in the C*-algebraic context, and one for {\qmss s}, where the situation is somewhat simpler and more elegant than in the general case.

We conclude with some examples. The most important of our examples is the last one, which shows that the Moyal plane, together with the so-called isospectral noncommutative geometry \cite{Varilly04}, is indeed a {\qmss}. The metric structure of the Moyal plane has attracted a lot of attention lately (e.g. \cite{Cagnache11,Martinetti11,Varilly04}), so our work brings a new component to this active area of research. Our example section also includes the basic but fundamental examples of locally compact metric spaces, compact quantum metric spaces, and all {\qms} with finite diameter fitting our work in \cite{Latremoliere05b}, such as \cite{Bellissard10}. 

We refer the reader to the reference book of G. Pedersen \cite{Pedersen79} for all the basic definitions about C*-algebras, their unitizations, and their state space.

\section{The noncommutative Monge-Kantorovich Distance}

This section introduces various substructures involved in our final definition of a {\qms}. The following notation will be used throughout this paper:
\begin{notation}
Let $\A$ be a C*-algebra. The norm of $\A$ is denoted $\|\cdot\|_\A$ and the state space of $\A$ is denoted by $\StateSpace(\A)$. The set of self-adjoint elements of $\A$ is denoted by $\sa{\A}$.
\end{notation}

\subsection{Lipschitz Pairs}

At the root of our work is a pair $(\A,\Lip)$ of a C*-algebra and a seminorm $\Lip$ enjoying various properties. We start with the most minimal assumptions on $\Lip$ in this section and introduce the context of the rest of this paper.

\begin{notation}
Let $\A$ be a C*-algebra. The smallest unital C*-algebra containing $\A$, i.e. either $\A$ if $\A$ is unital, or its standard unitization $\A\oplus\C$ \cite{Pedersen79} otherwise, is denoted by $\unital{\A}$. The unit of $\unital{\A}$ is always denoted by $\unit_{\unital{\A}}$. Note that $\sa{\unital{\A}} = \sa{\A} \oplus \R\unit_{\unital{\A}}$ if $\A$ is not unital. Any state $\varphi$ of $\A$ has a unique extension $\varphi'$ to a state of $\unital{\A}$ with $\varphi'(\unit_{\unital{\A}})=1$, and we shall always identify $\varphi$ and $\varphi'$ in this paper without further mention. Under this identification, the state space of $\unital{\A}$ equals to the quasi-state space of $\A$, and the weak* topology $\sigma(\A^\ast,\A)$ on $\StateSpace(\A)$ agrees with the weak* topology $\sigma(\unital{\A}^\ast,\unital{\A})$ restricted to $\StateSpace(\A)$.
\end{notation}

\begin{definition}
A \emph{Lipschitz pair} $(\A,\Lip)$ is a C*-algebra $\A$ and a seminorm $\Lip$ defined on a dense subspace of $\sa{\unital\A}$ such that $\{ a \in \sa{\unital{\A}} : \Lip(a) = 0 \} = \R \unit_{\unital{\A}}$.
\end{definition}

The metric which will be the focus of all our attention is:

\begin{definition}
The \emph{{\mongekant}} associated to a Lipschitz pair $(\A,\Lip)$ is the extended metric defined on the state space $\StateSpace(\A)$ of $\A$ by:
\begin{equation*}
\Kantorovich{\Lip} : \varphi, \psi \in \StateSpace(\A) \longmapsto \sup \{ |\varphi(a) - \psi(a)| : a\in\sa{\A} \text{ and } \Lip(a)\leq 1 \} \text{.}
\end{equation*}
\end{definition}

\begin{remark}
Let $(\A,\Lip)$ be a Lipschitz pair. The symmetry and triangle inequality properties of the {\mongekant} are easy to establish. The fact $\Kantorovich{\Lip}(\varphi,\psi) = 0 \Rightarrow \varphi = \psi$ for any two states $\varphi,\psi\in\StateSpace(\A)$ follows from the density of the domain of $\Lip$. Thus, $\Kantorovich{\Lip}$ satisfies the axiom of a metric, except that it may take the value $\infty$. Hence the term ``extended metric''.
\end{remark}

This extended metric has a long history and many names. It would probably be fair to name it the Monge-Kantorovich-Rubinstein-Wasserstein-Dobrushin metric. It finds its origins in the transportation problem introduced and studied by Monge in 1781. In 1940, a first formulation for this metric was introduced by Kantorovich in \cite{Kantorovich40} motivated by Monge's transportation problem. The form we use was derived in 1958 by Kantorovich and Rubinstein in \cite{Kantorovich58}. Later, Wasserstein introduced this metric for probabilities over a compact metric space again in \cite{Wasserstein69}. This metric was then extended and studied by Dobrushin on non-compact spaces in \cite{Dobrushin70}, where the name Vasershtein metric first appeared (Vasershtein is an alternative spelling for the translation from Cyrillic to Latin alphabets for Wasserstein). We choose our naming convention based on the original appearance of this metric in one form or another and admit to some arbitrariness in the matter. Our choice for a name follows Rieffel's convention as well.

For any Lipschitz pair $(\A,\Lip)$, when working in the unitization $\unital{\A}$ of $\A$, rather than $\A$, one can translate elements $a$ with $\Lip(a)<\infty$ by any scalar multiple of $\unit_{\unital{\A}}$, without changing the value of $\Lip$. This allows us to give an often useful expression for {\mongekant} which will play a central role in our characterization of {\qms}.

\begin{notation}
Let $(\A,\Lip)$ be a Lipschitz pair and $\mu$ a state of $\A$. We set:
\begin{equation*}
\mulip(\A,\Lip,\mu) = \{ a \in \sa{\unital{\A}} : \Lip(a)\leq 1\text{ and } \mu(a) = 0  \} \text{.}
\end{equation*}
\end{notation}

\begin{proposition}\label{wasskant-alt-expression-prop}
Let $(\A,\Lip)$ be a Lipschitz pair and $\mu$ be a state of $\A$. For any $\varphi,\psi \in \StateSpace(\A)$ we have:
\begin{equation*}
\Kantorovich{\Lip}(\varphi,\psi) = \sup \{ |\varphi(a)-\psi(a)| : a \in \mulip(\A,\Lip,\mu) \} \text{.}
\end{equation*}
\end{proposition}

\begin{proof}
First, note that for all $a\in \sa{\unital{A}}$ with $\Lip(a)<\infty$, we have, for any $\lambda\in\R$:
\begin{align*}
\Lip(a) &= \Lip(a+\lambda\unit_{\unital{\A}}-\lambda\unit_{\unital{\A}}) \leq \Lip(a+\lambda\unit_{\unital{\A}}) + |\lambda|\Lip(\unit_{\unital{\A}})\\
&= \Lip(a+\lambda\unit_{\unital{\A}}) \leq \Lip(a)+|\lambda|\Lip(\unit_{\unital{\A}}) = \Lip(a)\text{.}
\end{align*}
Hence $\Lip(a)=\Lip(a+\lambda\unit_{\unital{\A}})$ for all $a\in\sa{\unital{\A}}$ with $\Lip(a)<\infty$ and $\lambda\in\R$. 
Thus, if $a\in\sa{\A}$ with $\Lip(a)\leq 1$ is given, then $a-\mu(a)\unit_{\unital{\A}} \in \mulip(\A,\Lip,\mu)$ (note that $\mu(a) \in \R$). Conversely, if $b \in \mulip(\A,\Lip,\mu)$ then $\Lip(b)\leq 1$ and $b = a + \lambda \unit_{\unital{\A}}$ for some $\lambda \in \R$ and $a \in \sa{\A}$ by definition of $\unital{\A}$. Since $\mu(b) = 0$ we have $\lambda = - \mu(a)$. Thus:
\begin{equation}
\{ a - \mu(a)\unit_{\unital{\A}} : a \in \sa{\A}\text{ and } \Lip(a)\leq 1 \} = \mulip(\A,\Lip,\mu) \text{.}
\end{equation}
Thus, using $\nu(\lambda\unit_{\unital{\A}})=\lambda$ for all $\lambda\in\C$ and $\nu\in\StateSpace(\A)$, we have:
\begin{align*}
\Kantorovich{\Lip}(\varphi,\psi) &=  \sup\{ |\varphi(a)-\psi(a)| : a\in\sa{\A}\text{ and }\Lip(a)\leq 1\}\\
&= \sup\{ |\varphi(a-\mu(a)\unit_{\unital{\A}})-\psi(a-\mu(a)\unit_{\unital{\A}})| : a\in\sa{\A}\text{ and }\Lip(a)\leq 1\}\\
&=  \sup\{ |\varphi(b)-\psi(b)| : b \in \mulip(\A,\Lip,\mu) \}\text{,}
\end{align*}
as desired.
\end{proof}

\begin{remark}
By abuse of notation, in examples, if $\A$ is not unital and $\Lip$ is a densely defined \emph{norm} on $\sa{\A}$, we may refer to $(\A,\Lip)$ as a Lipschitz pair with the implicit understanding that when necessary, we will work with the extension $\unital{\Lip}$ of $\Lip$ to $\sa{\unital{\A}}$ given by $\unital{\Lip}(a+\lambda \unit_{\unital{\A}}) = \Lip(a)$ for all $a\in\sa{\A}$ and $\lambda\in\R$ --- since $(\A,\unital{\Lip})$ is then indeed a Lipschitz pair. We will even abuse notation further by writing $\Lip$ for $\unital{\Lip}$. This obvious terminology extension will be carried out implicitly to all structures including a Lipschitz pair.
\end{remark}

\begin{remark}
Let $(\A,\Lip)$ be a Lipschitz pair. If $\A$ is not unital, the proof of Proposition (\ref{wasskant-alt-expression-prop}) also applies if $\mu$ is the state $\mu : a+\lambda\unit_\A \in \unital{\A} \mapsto \lambda$. Otherwise, $\A = \unital{\A}$, so either way we get:
\begin{equation*}
\Kantorovich{\Lip}(\varphi,\psi) = \sup \{ |\varphi(a)-\psi(a)| : a\in\sa{\unital{\A}} \text{ and } \Lip(a) \leq 1\} 
\end{equation*}
for all $\varphi,\psi \in \StateSpace(\A)$.
\end{remark}

The main matter of this paper is to study the topological properties of the {\mongekant}. We start with:

\begin{proposition}\label{wasskant-finer-than-weak-prop}
Let $(\A,\Lip)$ be a Lipschitz pair. The topology induced by the {\mongekant} $\Kantorovich{\Lip}$ on the state space $\StateSpace(\A)$ of $\A$ is finer than the weak* topology.
\end{proposition}

\begin{proof}
Assume $(\varphi_n)_{n\in\N}$ is a sequence of states of $\A$ converging to \emph{a state} $\varphi \in \StateSpace(\A)$ for $\Kantorovich{\Lip}$. Let $a\in \sa{\A}$. Let $\varepsilon>0$ be given. Since the domain of $\Lip$ is dense in $\sa{\unital{\A}}$, there exists $b \in \sa{\unital{\A}}$ such that $\|a-b\|_{\unital{\A}} \leq \frac{1}{3}\varepsilon$ and $\Lip(b)<\infty$. Let $c = (\max\{\Lip(b),1\})^{-1} b$. Then $\Lip(c) \leq 1$ and thus, by definition of $\Kantorovich{\Lip}$, we have $|\varphi_n(c)-\varphi(c)| \leq \Kantorovich{\Lip}(\varphi_n,\varphi)$. Let $N\in\N$ be chosen so that $\Kantorovich{\Lip}(\varphi_n,\varphi) < \frac{1}{3 \max\{\Lip(b),1\}} \varepsilon$ for all $n\geq N$. Then, for all $n\geq N$:
\begin{align*}
|\varphi_n(a)-\varphi(a)| &\leq |\varphi_n(a)-\varphi_n(b)| + |\varphi_n(b)-\varphi(b) | + |\varphi(b)-\varphi(a)| \\
&\leq \frac{1}{3}\varepsilon + (\max\{\Lip(b),1\})|\varphi_n(c)-\varphi(c)| + \frac{1}{3}\varepsilon \leq \varepsilon \text{.}
\end{align*}
Hence $(\varphi_n(a))_{n\in\N}$ converges to $\varphi(a)$ for all $a\in\sa{\A}$. Since every element $a$ of $\A$ can be written as $a = \Re(a)+i\Im(a)$ with $\Re(a) = \frac{a+a^\ast}{2}$ and $\Im(a)=\frac{a-a^\ast}{2i}$ self-adjoints, we conclude that $(\varphi_n)_{n\in\N}$ weak*-converges to $\varphi$, as desired.
\end{proof}

The topological nature of closed balls for the {\mongekant} is interesting and needs some care in general.

\begin{notation}
For any (extended) metric $m$ on a space $X$, any $x\in X$ and any nonnegative real number $r$, we denote the closed ball $\{ y \in X : m(x,y)\leq r\}$ for $m$ of center $x$ and radius $r$ by $\cBall{m}{x}{r}$.
\end{notation}

\begin{proposition}\label{closed-balls-prop}
Let $(\A,\Lip)$ be a Lipschitz pair. For any $\mu\in\StateSpace(\A)$ and $r \in [0,\infty) \subseteq \R$, the closed ball $\cBall{\Kantorovich{\Lip}}{\mu}{r}$ is closed in the relative topology induced by the weak* topology on $\StateSpace(\A)$.
\end{proposition}

\begin{proof}
Let $(\varphi_n)_{n\in\N}$ be a sequence in  $\cBall{\Kantorovich{\Lip}}{\mu}{r}$ weak* converging to some \emph{state} $\varphi \in \StateSpace(\A)$. Then, for all $a\in \sa{\A}$ with $\Lip(a)\leq 1$ and all $n \in\N$, we have:
\begin{equation*}
\begin{split}
|\varphi(a) - \mu(a)| &\leq |\varphi(a) - \varphi_n(a) | + |\varphi_n(a)-\mu(a) |\\ &\leq |\varphi(a)-\varphi_n(a)| + r \stackrel{n\rightarrow\infty} {\longrightarrow} r \text{.}
\end{split}
\end{equation*}
Hence $\varphi \in \cBall{\Kantorovich{\Lip}}{\mu}{r}$.
\end{proof}

It is important to note that closed balls for the {\mongekant} are not in general closed in the weak* topology --- Proposition (\ref{closed-balls-prop}) holds for the \emph{relative} topology of the weak* topology on the state space only. 

\begin{counterexample}
\emph{Closed balls for the Monge-Kantorovich metric are not weak* closed in general.} Indeed, let $\A$ be the C*-algebra $C_0((0,1))$ of complex-valued continuous functions on $[0,1]\subseteq \R$ vanishing at $0$ and $1$, and $\Lip$ be the usual Lipschitz seminorm associated to the standard distance on $(0,1)$. Then $(\A,\Lip)$ is a Lipschitz pair. It is easy to check that $\StateSpace(C_0((0,1))) = \cBall{\Kantorovich{\Lip}}{\delta_x}{1}$ for any $x\in (0,1)$, where $\delta_x$ is the Dirac probability measure at $x$. Yet $\StateSpace(C_0((0,1)))$ is not weak* closed as $C_0((0,1))$ is not unital (for instance $\left(\delta_{\frac{1}{n}}\right)_{n\in\N}$ converges to $0$).
\end{counterexample}

Another observation is that closed balls for the {\mongekant}, even when they are weak* compact, are not usually compact for the topology induced by the extended metric, and in general the latter topology does not agree with the trace of the weak* topology on the balls.

\begin{counterexample}\label{noncompact-closed-balls-cex}
\emph{The topology induced by {\mongekant} does not agree with the weak* topology on closed balls in general.} Let $\A = C_0(\R)$ be the C*-algebra of $\C$-valued continuous functions on $\R$. Let $\Lip$ be the usual Lipschitz seminorm for the standard distance on $\R$. For any $x\in \R$, let $\delta_x : f\in C_0(\R) \mapsto f(x)$ be the Dirac state at $x$. Then one checks easily:
\begin{equation}
\Kantorovich{\Lip} \left(\delta_0, \frac{1}{n}\delta_n + \frac{n-1}{n}\delta_0\right) = 1\text{,}
\end{equation}
yet $\left(\frac{1}{n}\delta_n + \frac{n-1}{n}\delta_0\right)_{n\in\N}$ converges in the weak* topology to $\delta_0$. Thus the closed ball $\cBall{\Kantorovich{\Lip}}{\delta_0}{1}$ is not $\Kantorovich{\Lip}$-compact, nor does the topology induced by $\Kantorovich{\Lip}$ agree with the relative topology induced from the weak* topology on $\cBall{\Kantorovich{\Lip}}{\delta_0}{1}$. However, it can be easily checked in this case that $\cBall{\Kantorovich{\Lip}}{\delta_0}{1}$ is weak* compact.
\end{counterexample}

\subsection{Topographic Quantum Spaces}

Another fundamental substructure for our purpose is the notion of a topographic quantum space, i.e. a C*-algebra where the notion of locality, and by extension a notion of approaching infinity, is defined by choosing a large enough commutative set of observables.

\begin{definition}\label{topographic-quantum-space-def}
A \emph{topographic quantum space} $(\A,\M)$ is a C*-algebra $\A$ and an Abel\-ian C*-subalgebra $\M$ such that $\M$ contains an approximate identity for $\A$. When $(\A,\M)$ is a topographic quantum space, the C*-algebra $\M$ is called the \emph{topography} of $(\A,\M)$.
\end{definition}

Our terminology is inspired by a pair of a separable Abelian C*-algebra $\A$ and a strictly positive element $h \in \A$ seen as a ``height function'', with, for each $r > 0$, the level set $h^{-1}([r,\|h\|_\A])$ being compact in the spectrum of $\A$, and the collection of these level sets creating a topographic map of the spectrum of $\A$. Our definition does not pick a strictly positive element in general so that it fits the non-separable case as well.

When working with Abelian C*-algebras, we will use the following notations.

\begin{notation}
Let $\M$ be an Abelian C*-algebra. The Gel'fand spectrum of $\M$, always assumed to be endowed with the weak* topology, is denoted by $\sigma(\M)$. For any subset $B$ of $\sigma(\M)$, we define the indicator function $\indicator{B}$ of $B$ as:
\begin{equation*}
\indicator{B} : x \in \sigma(\M) \mapsto \begin{cases}
1 \;\text{if $x \in B$,}\\
0 \;\text{otherwise.}
\end{cases}
\end{equation*}
Any state $\varphi$ of $\M$ is the integral against a uniquely defined Radon probability measure on $\sigma(\M)$, and this probability measure is still denoted by $\varphi$. Thus we shall simply write $\varphi(B)$ for $\varphi(\indicator{B})$ for any Borel subset $B$ of $\M$.
\end{notation}

We note that $\indicator{B} \in \M^{\ast\ast}$ for any Borel subset $B$ of the Gel'fand spectrum of an Abelian C*-algebra $\M$, where the topological bidual $\M^{\ast\ast}$ of $\M$ is endowed with its Von Neumann algebra structure \cite{Pedersen79}. Now, let $(\A,\M)$ be a topographic quantum space. The Von Neumann algebra $\M^{\ast\ast}$ is identified once and for all with the Von Neumann subalgebra of $\A^{\ast\ast}$ obtained by completing $\M\subseteq \A^{\ast\ast}$ with respect to the strong topology in $\A^{\ast\ast}$ \cite[3.7.8]{Pedersen79}. With this identification, for any topographic quantum space $(\A,\M)$, we thus note that for all Borel subset $B$ of $\sigma(\M)$ and $a\in \unital{A}$, we have $\indicator{B}a\indicator{B} \in \A^{\ast\ast}$. Moreover, by \cite[3.7.8]{Pedersen79}, every state of $\A$ defines  a unique normal state of $\A^{\ast\ast}$ (via $a\in \A^{\ast\ast} \mapsto a(\varphi)$), and we identify these two states in this paper without further mention. 

Let $(\A,\M)$ be a topographic quantum space. If $\A$ is not unital, then $\M \oplus \C \unit_{\unital{\A}} \subseteq \unital{\A}$ is *-isomorphic to $\unital{\M}$. On the other hand, if $\A$ is unital, then any approximate unit of $\A$ converges in norm to $\unit_{\unital{\A}}$, and thus $\unit_{\unital{\A}} \in \M$ since $\M$ is closed. Hence without any ambiguity, we will always use the following convention:

\begin{convention}
Let $(\A,\M)$ be a topographic quantum space. Then $\unit_{\unital{\M}} = \unit_{\unital{\A}}$ and $\unital{\M}$ is the unital Abelian C*-subalgebra $\M + \C \unit_{\unital{\A}} \subseteq \unital{\A}$ of $\unital{\A}$.
\end{convention}

The requirement of existence of an approximate unit in $\M$ for any topographic quantum space $(\A,\M)$ ensures non-degeneracy and that going to infinity in $\M$ can be used to go to infinity in $\A$. We shall often use some specific choices of approximate identities of $\A$ in $\M$, and the following easy corollary of Urysohn's Lemma  will provide these elements when needed. Note however that the existence of the elements provided by this next lemma relies heavily on the fact that $\M$ is Abelian, and at the same time, will be of central importance in the development of our theory. We also use this opportunity to introduce our notation for nets, and a notation for the directed set of compact subsets of a topological space, with inclusion as the dual order.

\begin{notation}
When choosing an arbitrary net $(a_\alpha)_{\alpha \in I}$, the default notation for the order on the directed set $I$ is $\succ$ and the directed set property is expressed with this notation as $\forall \alpha,\beta \in I \; \exists \gamma \in I \;\; \left(\gamma\succ \alpha \text{ and } \gamma\succ \beta\right)$.
\end{notation}

\begin{notation}
Let $X$ be a topological space. The set of all compact subsets of $X$ is denoted by $\compacts{X}$. We note that it is directed set by choosing $\succ$ as the dual order to the inclusion.
\end{notation}

\begin{lemma}\label{topographic-approx-identity-lem}
Let $(\A,\M)$ be a topographic quantum space. There exists a directed set $I$ and a net $(f_\alpha,K_\alpha)_{\alpha \in I}$ of elements in $\sa{\M}\times\compacts{\sigma(\M)}$ such that:
\begin{enumerate}
\item $\{ K_\alpha : \alpha \in I \} = \compacts{\sigma(\M)}$,
\item For all $\alpha\in I$ we have $f_\alpha \indicator{K_\alpha} = \indicator{K_\alpha}$,
\item If $\alpha,\beta \in I$ and $\alpha \succ \beta$ then $f_\alpha f_\beta = f_\beta$,
\item For all $\alpha \in I$, we have $0\leq f_\alpha\leq \unit_{\unital{\A}}$,
\item $(f_\alpha)_{\alpha\in I}$ is an approximate identity for $\A$.
\end{enumerate}
Note that in particular $f_\alpha$ has compact support in $\M$ for all $\alpha \in I$.
\end{lemma}

\begin{proof}
Before we prove the existence of our approximate unit of choice, we note that if $(f_\alpha)_{\alpha\in I}$ satisfies Assertions (3),(5) of our proposition, then for all $\alpha \in I$, the function $f_\alpha$ is compactly supported. Indeed, assume that $\M$ is not unital (otherwise all elements of $\M$ obviously have compact support). Assume first that there exists $\gamma \in I$ such that for all $\beta\in I$ we have $\gamma\succ \beta$. Let $g\in \M$. Since $\lim_{\alpha\in I} \|g-gf_\alpha\|_\A = 0$, by definition of convergence for nets, we conclude that $\|g-gf_\gamma\| = 0$. Consequently, $g=gf_\gamma = f_\gamma g$ and thus $f_\gamma = \unit_{\unital{\A}}$. This contradicts our assumption that $\M$ is not unital.  Thus, $\M$ not being unital implies that for all $\alpha \in I$, there exists $\beta\in I$ with $\beta\succ \alpha$. Then the requirement $f_\beta f_\alpha = f_\alpha$ implies that the support of $f_\alpha$ is contained in the compact set $f_\beta^{-1}(\{-1\})$ (note that $f_\beta$ vanishes at infinity on $\sigma(\M)$ as $\M \cong C_0(\sigma(\M))$, which is why $f_\beta^{-1}(\{-1\})$ is not only closed, but also compact).

We now turn to the construction of our approximate unit.

Since $\sigma(\M)$ is locally compact Hausdorff, for any compact subset $K$ of $\sigma(\M)$, there exists an open set $U$ of $\sigma(\M)$ such that $K\subseteq U$ and the closure $\overline{U}$ of $U$ is compact. Indeed, by definition of local compactness, for all $x \in K$, there exists an open neighborhood $U_x$ of $x$ with $\overline{U_x}$ compact, and since $K \subseteq \bigcup_{x\in K} U_x$ and $K$ is compact, there exists a finite subset $F\subseteq K$ with $K\subseteq U = \bigcup_{x\in F} U_x$. Of course $U$ is open as a union of open sets, but since $\overline{U} = \bigcup_{x\in F}\overline{U_x}$ and $F$ is finite, by construction $\overline{U}$ is compact.

Let $I$ be the set:
\begin{equation*}
I = \{ (K,U) : K \in \compacts{\sigma(\M)},\text{ $U$ is open in $\sigma(\M)$, } \overline{U}\in\compacts{\sigma(\M)}\text{ and } K\subseteq U\} \text{,}
\end{equation*}
and define the following relation on $I$:
\begin{equation*}
\forall (K,U),(C,V) \in I \quad\quad (K,U)\succ(C,V) \iff (K=C\text{ and } U=V)\vee(\overline{V}\subseteq K) \text{.}
\end{equation*}
By construction, $\succ$ is reflexive. Moreover, if $(K,U)\succ (C,V)$ and $(C,V) \succ (T,W)$ for any $(K,U)\not= (C,V) \not= (T,W) \in I$, then by definition of $\succ$, we have $\overline{W}\subseteq{C}$ and $\overline{V}\subseteq K$, while by definition of $I$, we have $C\subseteq V\subseteq\overline{V}$ so $\overline{W}\subseteq K$ i.e. $(K,U)\succ(T,W)$. Since transitivity is obvious if either $(C,V)=(T,W)$ or $(C,V) = (K,U)$ we conclude that $\succ$ is a preorder on $I$.

Moreover, let $(K,U),(C,V) \in I$. Let $Q = \overline{U}\cup\overline{V}$ and note that $Q$ is compact in $\sigma(\M)$ by definition of $I$. Hence there exists an open set $W$ of $\sigma(\M)$ with compact closure and such that $Q\subseteq W$. By definition, $(Q,W)\in I$ and $(Q,W)\succ(K,U)$, $(Q,W)\succ (C,V)$. Hence $(I,\succ)$ is a directed set. 

We now denote the first component of $\alpha \in I$ as $K_\alpha \in \compacts{\sigma(\M)}$. By construction, $\compacts{\sigma(\M)} = \{ K_\alpha : \alpha \in I \}$. 

Let $\alpha = (K_\alpha,U_\alpha)\in I$. By Urysohn's lemma for locally compact Hausdorff spaces \cite{Folland}, there exists a continuous function $f\in\sa{M}$ such that $f(x) = 1$ if and only if $x\in K$ and $f(x) \not=0$ if and only if $x\in U$ while $0\leq f(x)\leq 1$ for all $x\in \sigma(\M)$. In particular, $f$ is compactly supported. Call a choice of such a function $f_\alpha$. We thus have constructed a net $(K_\alpha,f_\alpha)_{\alpha \in I}$ which satisfies, by construction, all the required properties. Indeed, $f_\alpha\indicator{K_\alpha} = \indicator{K_\alpha}$ since $f(x)=1$ for $x\in K_\alpha$. If $\alpha\succ \beta$ for $\alpha = (K_\alpha,U_\alpha), \beta = (K_\beta,U_\beta) \in I$ then if $x\not\in U_\beta$ then $f_\beta(x) = 0$ and if $x\in U_\beta$ then $x\in K_\alpha$ so $f_\alpha(x)=1$; either way $f_\alpha(x)f_\beta(x) = f_\beta(x)$. 

We now show that $(f_\alpha)_{\alpha\in I}$ is an approximate unit for $\M$. Let $g \in \M$. Let $\varepsilon > 0$. Since $g$ vanishes at infinity in $\M \cong C_0(\sigma(\M))$, there exists a compact subset $K$ of $\sigma(\M)$ such that, for all $x\in \sigma(\M)\setminus K$ we have $|g(x)|\leq\frac{1}{2}\varepsilon$. Let $\alpha_\varepsilon = (K,U) \in I$ for some open set $U$ containing $K$ and with compact closure. Then for all $\alpha\succ \alpha_\varepsilon$, we have:
\begin{equation*}
\begin{split}
\|g-gf_\alpha\|_\M &\leq \|(g-gf_\alpha)\indicator{K}\|_{\M^{\ast\ast}} + \|(g-gf_\alpha)(\unit_{\unital{\M}}-\indicator{K})\|_{\M^{\ast\ast}}\\
&\leq 0 + 2\|g(\unit_{\unital{\M}}-\indicator{K})\|_{\M^{\ast\ast}} \leq \varepsilon \text{.}
\end{split}
\end{equation*}
Thus $\lim_{\alpha\in I} \|g-gf_\alpha\|_\A = 0$ as desired.

Last, since $(\A,\M)$ is a topographic quantum space, there exists an approximate identity $(e_\beta)_{\beta \in J}$ for $\A$ in $\M$. Let $a\in\A$ and $\varepsilon > 0$ be given. Then, since $(e_\beta)_{\beta\in J}$ is an approximate unit for $\A$, there exists $\beta_\varepsilon \in J$ such that for all $\beta\succ \beta_\varepsilon$ we have $\|a-ae_\beta\|_\A\leq \frac{1}{2}\varepsilon$. Now, there exists $\alpha_\varepsilon \in I$ such that for all $\alpha\in I$ with $\alpha\succ \alpha_\varepsilon$ we have $\|e_{\beta_\varepsilon}-e_{\beta_\varepsilon} f_\alpha\|_\A \leq (2\max\{\|a\|_\A,1\})^{-1} \varepsilon$. Thus, for $\alpha \succ \alpha_\varepsilon$:
\begin{equation*}
\begin{split}
\|a-af_\alpha\|_\A &= \|a(\unit_{\unital{\A}} - f_\alpha)\|_{\A}  \leq \|ae_{\beta_\varepsilon}(\unit_{\unital{\A}}-f_\alpha)\|_\A + \|(a-ae_{\beta_\varepsilon})(\unit_{\unital{\A}}-f_\alpha)\|_\A \\
&\leq \|a\|_\A\|e_{\beta_\varepsilon}(\unit_{\unital{\A}}-f_\alpha)\|_{\A} + \|a-ae_{\beta_\varepsilon}\|_\A\\
&\leq \varepsilon \text{,}
\end{split}
\end{equation*}
so $\lim_{\alpha\in I} \|a-af_\alpha\|_\A = 0$. Consequently, $(a_\alpha)_{\alpha\in I}$ is an approximate unit in $\A$.

This concludes the construction of our special approximate unit.
\end{proof}

We now turn our attention to the notion of a tight set. In classical probability theory, a subset $\mathscr{P}$ of Borel probability measures over a locally compact Hausdorff space $X$ has a weak* closure containing only Borel probability measures if and only if it is (uniformly) tight, namely when for all $\varepsilon > 0$ there exists a compact $K\subseteq X$ such that for all $\mu \in \mathscr{P}$ we have $\mu(X\setminus K) \leq \varepsilon$. A noncommutative analogue for topographic quantum spaces is given as follows:

\begin{definition}\label{tight-def}
Let $(\A,\M)$ be a topographic quantum space. A subset $\mathscr{K}$ of $\StateSpace(\A)$ is \emph{tight} when:
\begin{multline*}
\forall\varepsilon>0\quad\exists Q \in \compacts{\sigma(\M)}\quad\forall K\in \compacts{\sigma(\M)} \\  (Q\subseteq K)\implies \sup \{ \varphi(1-\indicator{K}) : \varphi \in \mathscr{K} \} < \varepsilon \text{.}
\end{multline*}
\end{definition}

\begin{theorem}\label{tight-eq-precompact-thm}
Let $(\A,\M)$ be a topographic quantum space. A subset $\mathscr{K}$ of $\StateSpace(\A)$ is tight if and only if its weak* closure is a weak* compact subset of $\StateSpace(\A)$.
\end{theorem}

\begin{proof}
Assume $\mathscr{K}$ is tight. Let $(\varphi_\alpha)_{\alpha\in I}$ be a net in $\mathscr{K}$ weak* converging to $\psi \in \A^\ast$. Note that $\psi$ is a continuous positive linear functional of norm at most $1$, so it is sufficient to show that $\|\psi\|_{\A^\ast} \geq 1$.

Let $a \in \M$ with $\|a\|_\A\leq 1$ . Then for all $\alpha \in I$ and $K\in\compacts{\sigma(\M)}$:
\begin{equation}
\begin{split}
|1-\psi(a)| &\leq |1-\varphi_\alpha(a)| + |\varphi_\alpha(a) - \psi(a)| \\
&\leq |\varphi_\alpha(\unit_{\unital{\A}} - \indicator{K})| + |\varphi_\alpha(\indicator{K}-a)| + |\varphi_\alpha(a) - \psi(a)| \text{.}
\end{split}
\end{equation}

Since $(\A,\M)$ is a topographic quantum space, Lemma (\ref{topographic-approx-identity-lem}) provides us with an approximate unit $(a_\beta)_{\beta \in J}$ for $\A$ in $\sa{\M}$ and a net $(K_\beta)_{\beta \in J}$ such that for all $\beta \in J$, we have $a_\beta\indicator{K_\beta} = \indicator{K_\beta}$ and $0 \leq a_\beta \leq 1$. Moreover, for all $K\in \compacts{\sigma(\M)}$ there exists $\beta \in J$ with $K=K_\beta$.

Let $\epsilon > 0$ and $\varepsilon = \min\{\epsilon,1\}$. Since $\mathscr{K}$ is tight, there exists $\beta \in J$ such that:
\begin{equation}
\sup_{\varphi\in\mathscr{K}}\varphi(\unit_{\unital{\A}} - \indicator{K_\beta})\leq \frac{1}{9}\varepsilon^2 \leq \frac{1}{3}\varepsilon \text{.}
\end{equation}
Moreover for all $\alpha \in I$, by Cauchy-Schwarz's Inequality, we have:
\begin{equation}
\begin{split}
|\varphi_\alpha(\indicator{K_\beta}-a_\beta)| &= |\varphi_\alpha(\indicator{K_\beta}a_\beta - a_\beta)| \leq \sqrt{\varphi_\alpha(\unit_{\unital{\A}} - \indicator{K_\beta})\varphi_\alpha(a_\beta^2)} \\ &\leq \sqrt{\varphi_\alpha(\unit_{\unital{\A}} - \indicator{K_\beta})}\leq \frac{1}{3}\varepsilon \text{.}
\end{split}
\end{equation}
Hence there exists $\beta\in J$ such that for all $\alpha \in I$ we have:
\begin{equation}
|1-\psi(a_\beta)| \leq \frac{2}{3}\varepsilon + |\varphi_\alpha(a_\beta))-\psi(a_\beta)|\text{.}
\end{equation}

Last, by weak convergence and since $a_\beta \in \A$, there exists $\alpha \in I$ such that $|\varphi_\alpha(a_\beta)-\psi(a_\beta)|\leq \frac{1}{3}\varepsilon$. Thus:
\begin{equation*}
\forall\varepsilon >0 \exists \beta \in J\;\; |1-\psi(a_\beta)|\leq \varepsilon\leq \epsilon \text{,}
\end{equation*}
so $\psi$ is a state of $\A$ since $(a_\beta)_{\beta\in J}$ is an approximate identity of $\A$.

Conversely, assume that the weak* closure of $\mathscr{K}$ is a weak* compact subset of $\StateSpace(\A)$. Since $\mathscr{K}$ is a subset of its closure, it is enough to assume $\mathscr{K}$ is weak* compact --- as a subset of a tight set is also tight.

Let $(a_\alpha)_{\alpha \in I}$ and $(K_\alpha)_{\alpha\in I}$ be given as in Lemma (\ref{topographic-approx-identity-lem}). For all $a\in \A$ we define 
\begin{equation*}
\Theta_{\mathscr{K}}(a) : \varphi \in \mathscr{K} \longmapsto \varphi(a)\text{.}
\end{equation*}

The map $\Theta_{\mathscr{K}}$ takes elements of $\A$ to complex-valued weak* continuous functions on $\mathscr{K}$, and, since $\mathscr{K}$ consists of states, i.e. positive linear functionals, $\Theta_{\mathscr{K}}$ is increasing on $\sa{\A}$.

On the other hand, since for all $\alpha,\beta\in I$ with $\alpha\succ\beta$ we have $a_\alpha a_\beta = a_\beta$, we see that for all $x\in \sigma(\M)$, if $a_\beta(x) \not=0$ then $a_\alpha(x) = 1$. Since $0\leq a_\alpha,a_\beta \leq \unit_{\unital{\A}}$, we see that $(a_\alpha)_{\alpha\in I}$ is an increasing net in $\sa{\A}$.

Therefore, $(\Theta_{\mathscr{K}}(\unit_{\unital{\A}} - a_\alpha))_{\alpha\in I}$ is a net of continuous functions on the compact $\mathscr{K}$ pointwise decreasing and pointwise convergent to the continuous function $0$ on $\mathscr{K}$. Hence by Dini's theorem, $(\Theta_{\mathscr{K}}(\unit_{\unital{\A}} - a_\alpha))_{\alpha\in I}$ uniformly converges to $0$. In other words:
\begin{equation}
\lim_{\alpha \in I} \sup \{ \varphi(\unit_{\unital{\A}}-a_\alpha) : \varphi \in \mathscr{K} \} = 0\text{.}
\end{equation}
 Now, let $\varepsilon > 0$. Let $\alpha \in I$ such that $\sup \{ |\varphi(\unit_{\unital{\A}}-a_\alpha)| : \varphi \in \mathscr{K} \} \leq \varepsilon$. By definition of $a_\alpha$, we have $\indicator{K_\beta} \geq a_\alpha$ for any $\beta \in I, \beta \succ \alpha$. Fix such a $\beta \in I, \beta \succ \alpha$. Then for any $K\in\compacts{\sigma(\M)}$ with $K_\beta\subseteq K$ we have $\unit_{\unital{\A}} - \indicator{K} \leq \unit_A - \indicator{K_\beta} \leq \unit_{\unital{\A}} - a_\alpha$, hence:
\begin{equation*}
\sup \{ \varphi(\unit_{\unital{\A}}-\indicator{K}) : \varphi \in \mathscr{K} \}  \leq \varepsilon \text{,}
\end{equation*}
as desired.
\end{proof}

We now can introduce a very important structure associated to a topographic quantum space: the local state space, i.e. a collection of states which, from the perspective of the topography, are indeed locally supported. The local state space will play a central role in our notion of {\qms s}.

\begin{definition}
Let $(\A,\M)$ be a topographic quantum space. Let $K$ be a compact subset of the Gel'fand spectrum $\sigma(\M)$ of $\M$. We define:
\begin{equation}
\begin{split}
\StateSpace(\A|K) &= \{ \varphi \in \StateSpace(\A) : [\varphi]_\M(K) = 1 \} \\
&= \{ \varphi \in\StateSpace(\A) : \varphi(\indicator{K}) = 1 \} \text{,}
\end{split}
\end{equation}

The \emph{local state space} $\StateSpace(\A|\M)$ of a quantum topographic space $(\A,\M)$ is the set:
\begin{equation*}
\bigcup \{ \StateSpace(\A|K) : K \in \compacts{\sigma(\M)} \}\text{,}
\end{equation*} 
whose elements will be called \emph{local states}.
\end{definition}

We observe that for all $K\subseteq \sigma(\M)$ compact, the set $\StateSpace(\A|K)$ is a weak* compact convex face of $\StateSpace(\A)$ \cite{Alfsen01} associated with the compactly supported projection $\indicator{K}$, in the sense of \cite{Akemann89}.

Our definition of a topographic quantum space ensures that the local state space is large, in fact norm dense, in the state space.

\begin{proposition}\label{local-state-space-dense-prop}
The local state space of a topographic quantum space $(\A,\M)$ is norm dense in the state space of $\A$.
\end{proposition}

\begin{proof}
Let $\varphi \in \StateSpace(\A)$ and $\epsilon > 0$. Let $\varepsilon = \min\left\{\frac{1}{2}\epsilon,\frac{1}{2}\right\}$. Since $\{\varphi\}$ is weak* compact in $\StateSpace(\A)$, it is a tight set, so, there exists $K_0 \in \compacts{\sigma(\M)}$ such that for all $K \in\compacts{\sigma(\M)}$ with $K_0 \subseteq K$, we have $1-\varepsilon^2 \leq \varphi(\indicator{K}) \leq 1$. Since $0\leq\varepsilon\leq\frac{1}{2}$, we thus have $|1-\varphi(\indicator{K})|\leq\varepsilon^2 \leq \varepsilon$ as well.

Let $K \in \compacts{\sigma(\M)}$, with $K_0\subseteq K$. Let $\psi_K : a \in \A \mapsto \varphi(\indicator{K})^{-1}\varphi(\indicator{K} a \indicator{K})$ (note that $\varphi(\indicator{K}) > 0$). Then by construction, $\psi_K(\indicator{K})=1$ and $\psi_K$ is a positive functional on $\unital{A}$, with $\|\psi_K\|_{\A^{\ast}} = \psi_K(\unit_{\unital{\A}}) = \psi_K(\indicator{K}) = 1$. Hence $\psi_K \in \StateSpace(\A|K)$. 

On the other hand, let $a\in \A$ be given. Then:
\begin{equation}\label{local-state-space-dense-prop-eq0}
\begin{split}
|\varphi(a)-\psi_K(a)| &\leq |\varphi(a)-\varphi(\indicator{K})^{-1}\varphi(a)| + \varphi(\indicator{K})^{-1}|\varphi(a)-\varphi(\indicator{K}a\indicator{K})| \\
&\leq |\varphi(a)|(\varphi(\indicator{K})^{-1}-1) + \varphi(\indicator{K})^{-1}|\varphi(a-\indicator{K} a \indicator{K})| \\
&\leq  \frac{\varepsilon}{1-\varepsilon} |\varphi(a)| + \frac{1}{1-\varepsilon}\sqrt{\varphi(1-\indicator{K})}\|a\|_\A \,\text{by Cauchy-Schwarz,}\\
&\leq  \frac{2\varepsilon}{1-\varepsilon}\|a\|_\A
\leq  \epsilon \|a\|_\A \text{.}
\end{split}
\end{equation}
Hence $(\psi_K)_{K\in \compacts{\sigma(\M)}}$ converges to $\varphi$ in norm.
\end{proof}

We conclude this section with a notation and the properties of the restriction map for linear functionals from $\A$ to $\M$ for a topographic quantum space $(\A,\M)$, as we shall use these facts later in this paper.

\begin{notation}
Let $\A$ be a C*-algebra and $\M$ be a C*-subalgebra of $\A$. For any continuous linear functional $\mu$ of $\A$, we denote by $[\mu]_\M$ the restriction of $\mu$ to $\M$. Thus $[\mu]_\M \in \M^\ast$ for all $\mu \in \A^\ast$.
\end{notation}

\begin{proposition}\label{restriction-prop}
Let $(\A,\M)$ be a topographic quantum space. Then:
\begin{enumerate}
\item If $\varphi \in \StateSpace(\A)$ then $[\varphi]_\M \in \StateSpace(\M)$,
\item If $\psi \in \StateSpace(\M)$ then there exists $\varphi \in \StateSpace(\A)$ such that $[\varphi]_\M = \psi$.
\item If $(\varphi_\alpha)_{\alpha \in I}$ is a net in $\StateSpace(\A)$ weak* converging in $\A^\ast$ to $\mu$ then $([\varphi_\alpha])_{\alpha\in I}$ weak* converges in $\M^\ast$ to $[\mu]_\M$. 
\end{enumerate} 
In other words, the map $\varphi \in \StateSpace(\A) \mapsto [\varphi]_\M$ is a well-defined weak*-continuous affine surjection onto $\StateSpace(\M)$.
\end{proposition}

\begin{proof}
Let $(e_\beta)_{\beta \in J}$ be an approximate unit for $\A$ in $\M$, which exists by Definition (\ref{topographic-quantum-space-def}). Let $\varphi \in \StateSpace(\A)$. Note that by definition, $[\varphi]_\M$ is a positive linear functional and that $\|[\varphi]_\M\|_{\M^\ast} \leq \|\varphi\|_{\A^\ast} = 1$. Then $1 \geq [\varphi]_\M (e_\beta) = \varphi(e_\beta) \stackrel{\beta \in J}{\longrightarrow} 1$, and thus $[\varphi]_\M = 1$ so $[\varphi]_\M \in\StateSpace(\M)$.

Conversely, if $\psi \in \StateSpace(\M)$,then by the Hahn-Banach extension theorem for positive linear functional, there exists $\varphi \in \A^\ast$ positive linear functional on $\A$ such that $[\varphi]_\M = \psi$. It is easy to see that $\varphi$ is indeed a state of $\A$.


The weak* continuity of this surjection is straightforward.
\end{proof}

\subsection{Lipschitz triples}

We now bring together the two substructures defined in this paper so far into an object whose signature will be the same as {\qms s}. In essence, a {\qms} will be a Lipschitz triple, as defined below, with an additional topological condition based on the notion of tame sets, defined in this section as well.

\begin{definition}
A \emph{Lipschitz triple} $(\A,\Lip,\M)$ is a Lipschitz pair $(\A,\Lip)$ and an Abelian C*-subalgebra $\M$ of $\A$ such that $(\A,\M)$ is a quantum topographic space.
\end{definition}

A Lipschitz triple allows us to define our noncommutative analogue of a Do\-bru\-shin-tight set (see Equation (\ref{Dobrushin-tight-intro} in the introduction), which we call a tame set of states.

\begin{definition}\label{tame-def}
Let $(\A,\Lip,\M)$ be a Lipschitz triple. A subset $\mathscr{K}$ of $\StateSpace(\A)$ is called \emph{$(\A,\Lip,\M)$-tame} when, for some $\mu \in \StateSpace(\A|\M)$:
\begin{equation*}
\lim_{K \in \compacts{\sigma(\M)}} \sup \{ |\varphi(a-\indicator{K}a\indicator{K})| : {a \in \mulip(\A,\Lip,\mu)}, \varphi \in \mathscr{K} \}= 0 \text{.}
\end{equation*}
\end{definition}

As a first observation, we note that the union of all tame sets of $\StateSpace(\A)$ for a Lipschitz triple $(\A,\Lip,\M)$ is norm dense in $\StateSpace(\A)$ since it contains the local state space:

\begin{proposition}\label{obvious-tame-prop}
Let $(\A,\Lip,\M)$ be a Lipschitz triple. For any $K\in \compacts{\sigma(\M)}$, the set $\StateSpace(\A|K)$ is tame.
\end{proposition}

\begin{proof}
Let $K\in\compacts{\sigma(\M)}$, and let $C\in\compacts{\sigma(\M)}$ with $K\subseteq C$. By Cauchy-Schwarz inequality, for all $\varphi \in \StateSpace(\A|K)$ and for all $a\in \unital{\A}$, we have:
\begin{equation*}
\begin{split}
|\varphi(a)-\varphi(\indicator{C}a\indicator{C})| &\leq |\varphi(\indicator{C}a(\unit_{\unital{\A}}-\indicator{C}))| + |\varphi((\unit_{\unital{\A}}-\indicator{C})a\indicator{C})|\\
&\quad + |\varphi((\unit_{\unital{\A}} - \indicator{C})a(\unit_{\unital{\A}}-\indicator{C}))|\\
&\leq 3\sqrt{\varphi(\unit_{\unital{\A}}-\indicator{C})}\|a\|_{\unital{\A}}
\leq 3\sqrt{\varphi(\unit_{\unital{\A}} - \indicator{K})}\|a\|_{\unital{\A}} = 0 \text{.}
\end{split}
\end{equation*}
Hence our proposition follows by definition.
\end{proof}

We now prove a very important result: tame sets are always tight. This is very useful for our purpose, and also shows that our notion is somewhat different from Dobrushin's notion in a topological sense. Indeed, in a finite diameter metric space, any set of probability measures is Dobrushin-tight, including the whole state space, while tame sets must have, among other properties, weak* closures contained in $\StateSpace(\A)$ --- which exclude the state space for any non-unital C*-algebra. On the other hand, tightness will enable us to derive key properties of tame sets in this section.

\begin{theorem}\label{tame-implies-tight-thm}
Let $(\A,\Lip,\M)$ be a Lipschitz triple. Then a tame subset of $\StateSpace(\A)$ is tight. In particular, the weak* closure of a tame subset is a weak* compact subset of $\StateSpace(\A)$.
\end{theorem}

\begin{proof}
Let $\mathscr{K}$ be a tame subset of $\StateSpace(\A)$. Let $\varepsilon > 0$. Let $U\subseteq \sigma(\M)$ be a nonempty open set with compact closure and $x\in U$. Since $\{x\}$ is compact, there exists by Urysohn's Lemma for locally compact Hausdorff spaces a continuous function $f\in\sa{\M}$ with $0\leq f\leq \unit_{\unital{\A}}$, $f(x)=1$ and $f\indicator{\sigma(\M)\setminus U} = 0$. Let $g = \unit_{\unital{\A}} - f \in \sa{\unital{M}}$ and note that $(\unit_{\unital{\A}}-\indicator{K})g = \unit_{\unital{\A}}-\indicator{K}$ for all $K\in\compacts{\sigma(\M)}$ with $U\subseteq K$.

Now, since $\{a \in \sa{\unital{\A}} : \Lip(a)<\infty \}$ is norm dense in $\sa{\unital{\A}}$, there exists $b \in \sa{\unital{\A}}$ with $\Lip(b) < \infty$ and $\| g - b\|_{\unital{\A}} < \frac{1}{3}\varepsilon$. 

Let $\varphi \in \StateSpace(\A)$ and set $\lambda = \max\{\Lip(b),1\} \in [1,\infty)\subseteq \R$. Let $K\in\compacts{\sigma(\M)}$ with $U\subseteq K$. Then, using the observation that $g$ commutes with the projection $\indicator{K}$:
\begin{equation}\label{tame-implies-tight-thm-eq0}
\begin{split}
|\varphi(\unit_A - \indicator{K})| &= |\varphi((\unit_{\unital{\A}}-\indicator{K})g)| \\
&= |\varphi(g - \indicator{K}g)| = |\varphi(g-\indicator{K}g\indicator{K})| \\
&\leq |\varphi(b-\indicator{K}b\indicator{K})| + |\varphi(b-g)| + |\varphi(\indicator{K}(g-b)\indicator{K})| \\
&\leq \lambda |\varphi(\lambda^{-1}b-\indicator{K}\lambda^{-1}b\indicator{K})| + \frac{2}{3}\varepsilon \text{.}
\end{split}
\end{equation}

Since $\mathscr{K}$ is tame, there exists $K_0 \in \compacts{\sigma(\M)}$ and $\mu \in \StateSpace(\A|\M)$ so that for all $K\in\compacts{\sigma(\M)}$ with $K_0\subseteq K$ we have:
\begin{equation}\label{tame-implies-tight-thm-eq1}
\sup \{ |\varphi(a-\indicator{K}a\indicator{K})| : \varphi \in\mathscr{K}\text{ and } a \in \mulip(\A,\Lip,\mu) \} \leq \frac{1}{3\lambda} \varepsilon \text{.}
\end{equation}

Thus, since $\Lip(\lambda^{-1} b) \leq 1$, we conclude from both Inequalities (\ref{tame-implies-tight-thm-eq0}) and (\ref{tame-implies-tight-thm-eq1}) that for all $K \in \compacts{\sigma(\M)}$ with $\overline{U} \cup K_0 \subseteq K$ we have:
\begin{equation}
\begin{split}
\sup \{ \varphi(\unit_{\unital{\A}} - \indicator{K}) : \varphi \in \mathscr{K} \} \leq \varepsilon\text{,}
\end{split}
\end{equation}

as desired.\end{proof}

In the classical situation, the set of all probability measures supported on a given compact in a metric space is always of finite diameter for the {\mongekant}. This is not true in general in the noncommutative setting (see the Examples section). However, the purpose of introducing the local state space is that we may require it to be well-behaved in the following sense:

\begin{definition}
A Lipschitz triple $(\A,\Lip,\M)$ is \emph{regular} when for all $K \in \compacts{\sigma(\M)}$ the set $\StateSpace(\A|K)$ has finite diameter for the {\mongekant} metric $\Kantorovich{\Lip}$ associated with $(\A,\Lip)$. 
\end{definition}

As a first step, we characterize regularity for Lipschitz triples in C*-algebraic terms, and we start doing so by establishing a couple of useful lemmas which we will need again later.

\begin{lemma}\label{pnorm-eq-norm-lem}
Let $\A$ be a C*-algebra and $p \in \A^{\ast\ast}$ be a projection. Then:
\begin{equation*}
\forall a \in \sa{\unital{\A}} \;\; \|pap\|_{\A^{\ast\ast}} = \sup \{ |\varphi(a)| : \varphi\in\StateSpace(\A)\text{ and } \varphi(p) = 1 \} \text{.}
\end{equation*}
\end{lemma}

\begin{proof}
Denote $\sup \{ |\varphi(a)| : \varphi\in\StateSpace(\A)\text{ and } \varphi(p) = 1 \}$ by $\pnorm{p}(a)$ for all $a\in \unital{\A}$. First, note that if $\varphi(p)=1$ then by Cauchy-Schwarz, for all $a\in \unital{\A}$:
\begin{equation*}
\begin{split}
|\varphi(a)| &\leq |\varphi(pap)| + |\varphi(pa(\unit_{\unital{\A}}-p))| + |\varphi((\unit_{\unital{\A}}-p)ap))|\\
&\quad +|\varphi((\unit_{\unital{\A}}-p)a(\unit_{\unital{\A}}-p))| \\
&\leq |\varphi(pap)| + 3\sqrt{\varphi(\unit_{\unital{\A}}-p)}\|a\|_{\unital{\A}} \\
&=  |\varphi(pap)| \leq \|pap\|_{\A^{\ast\ast}}
\end{split}
\end{equation*}
so $\pnorm{p}(a) \leq \|pap\|_{\A^{\ast\ast}}$ for all $a\in \A$.

Let $\psi \in \StateSpace(\A)$. If $\psi(p)= 0 $ then $a\in\A \mapsto \psi(pap) = 0$ by Cauchy-Schwarz, so $\psi(pap) \leq \pnorm{p}(a)$ for all $a\in \unital{\A}$. If instead $\psi(p) \in (0,1]$ then $\psi' : a \in \unital{\A} \mapsto \psi(p)^{-1}\psi(pap)$ is a state of $\unital{\A}$ with $\psi'(p)=1$, as for all $a\in\unital{\A}$, $a\geq 0 \implies \indicator{K}a\indicator{K} \geq 0$ so $\psi'$ is a positive functional on $\unital{\A}$ of norm $\psi'(\unit_{\unital{\A}})= \psi'(p) = 1$. Thus if $\psi(p)>0$ then $\pnorm{p}(a) \geq \psi'(a) \geq \psi(pap)$ for all $a\in \unital{\A}$. 
Thus we always have, for all $a\in \A$ and any state $\psi$, that $\psi(pap)\leq \pnorm{p}(a)\leq \|pap\|_{\A^{\ast\ast}}$. Hence for $a\in\sa{\unital{\A}}$, since $pap \in \sa{\A^{\ast\ast}}$, we have \cite[3.7.8]{Pedersen79}:
\begin{equation*}
\|pap\|_{\A^{\ast\ast}} = \pnorm{p} (a) \text{.}
\end{equation*}
\end{proof}

The following lemma is from \cite{Kadison91}, though we include the proof for the convenience of the reader.

\begin{lemma}\label{increasing-norms-lem}
Let $\A$ be a C*-algebra, and let $a,s,t,u,v \in \A$. Assume that $0\leq s\leq u$ and $0 \leq t\leq v$ in $\A$. Then $\|\sqrt{s}a\sqrt{t}\|_\A \leq \|\sqrt{u}a\sqrt{v}\|_\A$.
\end{lemma}

\begin{proof}
See \cite[Exercise 4.6.1]{Kadison91}. Note that:
\begin{equation*}
0\leq (\sqrt{s}a)^\ast (\sqrt{s}a) = a^\ast s a \leq a^\ast u a = (\sqrt{u}a)^\ast (\sqrt{u}a) 
\end{equation*}
so $\|\sqrt{s}a\|_\A \leq \|\sqrt{u}a\|_\A$. Similarly $\|a\sqrt{t}\|_\A \leq \|a\sqrt{v}\|_\A$. Hence:
\begin{equation*}
\|\sqrt{s}a\sqrt{t}\|_\A \leq \|\sqrt{u}a\sqrt{t}\|_\A \leq \|\sqrt{u}a\sqrt{v}\|_\A \text{.}
\end{equation*} 
\end{proof}

We shall use Lemma (\ref{increasing-norms-lem}) when $u,v,s,t$ are projections, in which case they are equal to their square roots.

\begin{proposition}\label{regular-characterization-prop}
Let $(\A,\Lip,\M)$ be a Lipschitz triple and $\mu \in \StateSpace(\A|\M)$. Then $(\A,\Lip,\M)$ is regular if and only if for all $K\in\compacts{\sigma(\M)}$, there exists $r_K \in \R$ such that:
\begin{equation*}
\sup \{ \|\indicator{K}a\indicator{K}\|_{\A^{\ast\ast}} : a \in \mulip(\A,\Lip,\mu) \} \leq r_K\text{.}
\end{equation*}
\end{proposition}

\begin{proof}
Assume that $(\A,\Lip,\M)$ is a regular Lipschitz triple. Let $\mu\in\StateSpace(\A|\M)$ and let $K\in\compacts{\sigma(\M)}$ . Let $K'\in\compacts{\sigma(\M)}$ such that $\mu \in \StateSpace(\A|K')$ and $K\subseteq K'$. Since $(\A,\Lip,\M)$ is regular, the set $\StateSpace(\A|K')$ has finite diameter $D\in \R$ for $\Kantorovich{\Lip}$. By Proposition (\ref{wasskant-alt-expression-prop}), Lemma (\ref{increasing-norms-lem}) and Lemma (\ref{pnorm-eq-norm-lem}), we have for all $a\in \mulip(\A,\Lip,\mu)$ that:
\begin{equation*}
\begin{split}
\|\indicator{K}a\indicator{K}\|_{\A^{\ast\ast}} &\leq \|\indicator{K'}a\indicator{K'}\|_{\A^{\ast\ast}} = \sup \{ |\varphi(a)| : \varphi \in \StateSpace(\A|K') \} \\
&= \Kantorovich{\Lip}(\varphi, \mu) \leq D \text{.}
\end{split}
\end{equation*}

Conversely, assume that for all $K\in \compacts{\sigma(\M)}$, there exists $r_K$ such that:
\begin{equation*}
\sup \{ \|\indicator{K}a\indicator{K}\|_{\A^{\ast\ast}} : a\in \mulip(\A,\Lip,\mu) \} \leq r_K\text{.}
\end{equation*}
Then by Proposition (\ref{wasskant-alt-expression-prop}) and by Lemma (\ref{pnorm-eq-norm-lem}), we have for all $\varphi \in \StateSpace(\A|K)$:
\begin{equation*}
\Kantorovich{\Lip}(\varphi,\mu) \leq r_K
\end{equation*}
and thus, by the triangle inequality for $\Kantorovich{\Lip}$, the set $\StateSpace(\A|K)$ has finite diameter for $\Kantorovich{\Lip}$.
\end{proof}

Regularity has two very important consequences on tame sets. First, the notion of a tame set does not depend on the choice of a particular local state. Second, tame sets are always contained in closed balls of finite radii around any local state. Both of these properties rely on regularity and the fact that tame implies tight.

\begin{theorem}\label{regular-lipschitz-tame-thm}
Let $(\A,\Lip,\M)$ be a regular Lipschitz triple. A subset $\mathscr{K}$ of $\StateSpace(\A)$ is $(\A,\Lip,\M)$-tame if and only if for all $\mu \in \StateSpace(\A|\M)$, we have:
\begin{equation*}
\lim_{K \in \compacts{\sigma(\M)}} \sup \{ |\varphi(a-\indicator{K}a\indicator{K})| : {a \in \mulip(\A,\Lip,\mu)}, \varphi \in \mathscr{K} \}= 0 \text{.}
\end{equation*}
\end{theorem}

\begin{proof}
The condition is sufficient by definition. Let us show that it is necessary. Assume $\mathscr{K}$ is $(\A,\Lip,\M)$-tame. Thus by Definition (\ref{tame-def}), there exists a local state $\nu$ of $(\A,\Lip,\M)$ such that:
\begin{equation}\label{regular-lipschitz-prop-eq0}
\lim_{K\in \compacts{\sigma(\M)}} \sup \{ |\varphi(a-\indicator{K}a\indicator{K}))| : {a \in \mulip(\A,\Lip,\nu)},\varphi \in \mathscr{K} \}= 0 \text{.}
\end{equation}
Let $\mu \in \StateSpace(\A|\M)$. By definition, there exists $K_0,K_1 \in \compacts{\sigma(\M)}$ such that $\mu \in \StateSpace(\A|K_0)$ and $\nu\in\StateSpace(\A|K_1)$. Thus $\mu, \nu \in S(\A,K)$ where $K = K_0\cup K_1 \in \compacts{\sigma(\M)}$.

Since $(\A,\Lip,\M)$ is regular, the set $\StateSpace(\A|K)$ has finite diameter for $\Kantorovich{\Lip}$. Thus by definition of $\Kantorovich{\Lip}$, there exists $r \in [0,\infty)\subseteq \R$ such that:
\begin{equation}\label{regular-lipschitz-prop-eq1}
\sup \{ |\mu(a)-\nu(a)| : a \in \sa{\unital{\A}}\text{ and } \Lip(a)\leq 1 \} = r \text{.}
\end{equation}
Now, for all $a\in \sa{\unital{\A}}$ with $\Lip(a)\leq 1$ and $\mu(a) = 0$, and for any $\varphi \in K$, we have:
\begin{align*}
|\varphi(a - \indicator{K} a \indicator{K})| &\leq |\varphi((a-\nu(a)\unit_{\unital{\A}})-\indicator{K}(a-\nu(a)\unit_{\unital{\A}})\indicator{K})|\\ 
&\quad +|\nu(a)||\varphi(\unit_{\unital{\A}}-\indicator{K})|\text{.}
\end{align*}
Hence:
\begin{equation}\label{regular-lipschitz-prop-eq2}
\begin{split}
0 &\leq \sup \{ |\varphi(a-\indicator{K}a\indicator{K}) | : \varphi \in \mathscr{K}\text{ and } a \in \mulip(\A,\Lip,\mu) \} \\ 
&\leq \sup\{ |\varphi(a-\indicator{K}a\indicator{K})| : \varphi \in \mathscr{K} \text{ and } a \in \mulip(\A,\Lip,\nu) \} \\
&\quad + r\sup \{ |\varphi(\unit_{\unital{\A}} - \indicator{K})| : \varphi \in \mathscr{K} \} \text{.}
\end{split}
\end{equation}
Since $\mathscr{K}$ is tame, it is tight, so $\lim_{K\in\compacts{\sigma(\M)}}\sup \{ \varphi(1-\indicator{K}) : \varphi \in K \} = 0$. Hence:
\begin{equation*}
\lim_{\alpha \in I} \sup \{ |\varphi(a - \indicator{K}a\indicator{K})| : {a \in \mulip(\A,\Lip,\mu)}\text{ and } \varphi \in \mathscr{K} \}= 0 \text{,}
\end{equation*}
as desired.
\end{proof}

\begin{proposition}\label{tame-bounded-prop}
Let $(\A,\Lip,\M)$ be a regular Lipschitz triple and $\mu \in \StateSpace(\A|\M)$. If $\mathscr{K}$ is a tame subset of $\StateSpace(\A)$ then there exists $r_{\mu,\mathscr{K}} \in [0,\infty)$ such that $\mathscr{K} \subseteq \cBall{\Kantorovich{\Lip}}{\mu}{r_{\mu,\mathscr{K}}}$.
\end{proposition}

\begin{proof}
Let $\mu$ be a local state and $\mathscr{K}$ be a tame set. By Proposition (\ref{wasskant-alt-expression-prop}), we have:
\begin{equation*}
\Kantorovich{\Lip}(\varphi,\mu) = \sup \{ |\varphi(a)| : a\in \mulip(\A,\Lip,\mu) \} \text{.}
\end{equation*}
Now for all $\varphi \in \mathscr{K}$ and $K\in\compacts{\sigma(\M)}$, we have:
\begin{equation}
\begin{split}
|\varphi(a)| &\leq |\varphi(\indicator{K}a\indicator{K})| + |\varphi(a)-\varphi(\indicator{K}a\indicator{K})|\text{.}
\end{split}
\end{equation}

Since $(\A,\Lip,\M)$ is regular, for any $K\in\compacts{\sigma(\M)}$, there exists $r_K \in \R$ such that for all $a\in\mulip(\A,\Lip,\mu)$ we have $\|\indicator{K}a\indicator{K}\|_{\A^{\ast\ast}} \leq r_K$. On the other hand, since $\mathscr{K}$ is tame, there exists $K_0 \in\compacts{\sigma(\M)}$ such that:
\begin{equation}
\sup \{ |\varphi(a - \indicator{K_0}a\indicator{K_0})| : a \in \mulip(\A,\Lip,\mu)\text{ and } \varphi\in\mathscr{K} \} \leq 1 \text{.}
\end{equation}

Hence:

\begin{equation*}
\sup \{ |\varphi(a)| : a \in \mulip(\A,\Lip,\mu)\text{ and } \varphi \in \mathscr{K} \} \leq r_{K_0} + 1
\end{equation*}
which completes our proof.
\end{proof}

We are now ready to define the main object of study of this paper.

\section{Quantum Locally Compact Spaces}

We start this section with the definition of a {\qms}, which is a Lipschitz triple for which the {\mongekant} is topologically well-behaved. We also define a pointed {\qms}, as it will be the main object of study for quantum Gromov-Hausdorff convergence in \cite{Latremoliere12c}. We then turn to proving characterizations of {\qms s}, since the definition itself can prove challenging to establish directly. Our characterizations rely on the introduction of a locally convex topology {\tut} on the C*-algebra component of a {\qms} which will characterize {\qms s} as those Lipschitz triples for which the Lipschitz ball will be {\tut}-totally bounded. The topology {\tut} depends on the topography, unlike the weakly uniform topology used for a similar purpose in \cite{Latremoliere05b}.

\subsection{Definition}

\begin{definition}\label{qms-def}
A regular Lipschitz triple $(\A,\Lip,\M)$ is a \emph{\qms} when, for all tame subsets $\mathscr{K}$ of the state space $\StateSpace(\A)$, the topology of the metric space $(\mathscr{K},\Kantorovich{\Lip})$ is the relative topology induced by the weak* topology restricted on $\mathscr{K}$, where we denoted the {\mongekant} associated with $(\A,\Lip)$ by $\Kantorovich{\Lip}$.

A \emph{\qmss} is a {\qms} $(\A,\Lip,\M)$ where $\A$ is a separable C*-algebra.

If $(\A,\Lip,\M)$ is a {\qms} then $\Lip$ is called a \emph{Lip-norm}.
\end{definition}

We shall see in the Examples section below that locally compact metric spaces, compact quantum metric spaces as defined by Rieffel, and bounded separable quantum locally compact metric spaces as we defined, tentatively, in \cite{Latremoliere05b}, all fit within our notion of {\qms}. We also wish to contrast our notion with the very interesting concept of a $W^\ast$-metric spaces introduced by Kuperberg and Weaver in \cite{Kuperberg10}. We thank the referee of our present paper for introducing us to \cite{Kuperberg10}. The notion of a $W^\ast$-metric space is inspired by the standard approach to quantum error correction. A $W^\ast$-metric on a Von Neumann algebra $\mathfrak{V}$ acting on some Hilbert space $\mathscr{H}$ is given by a $W^\ast$-filtration $(\mathfrak{V}_s)_{s\in [0,\infty)}$ with $\mathfrak{V}_0 = \mathfrak{V}'$, where a $W^\ast$-filtration $(\mathfrak{V}_s)_{s > 0}$ is a one parameter family of dual operator systems such that $\mathfrak{V}_t\mathfrak{V}_s \subseteq \mathfrak{V}_{s+t}$ and $\mathfrak{V}_t = \bigcap_{r>t}\mathfrak{V}_r$ for all $s,t \in [0,\infty)$ (see \cite[Definition 2.1]{Kuperberg10}). This notion allows to define a distance between projections of the Von Neumann algebra $\mathfrak{V}$, and thus in turn, Kuperberg and Weaver can define Lipschitz elements, uniformly continuous elements, and many other central notions from metric space theory, which they apply to many contexts. Spectral triples naturally define $W^\ast$-metrics, and $W^\ast$-metrics in turn, naturally define Lipschitz Leibniz seminorms in the sense of \cite{Rieffel00}. 

A common thread to our notion and the notion of a $W^\ast$-metric may be the ability to define forms of locality. In \cite[Definition 3.17]{Kuperberg10}, a notion of local convergence between $W^\ast$-metrics is introduced, where the filtrations are used to provide a mean to ``slice'' the quantum space and approximate each slice. Our research into the notion of Gromov-Hausdorff convergence for a class of {\qms s} suggests, as hinted in many proofs in this paper, the use of a notion of ``slices'' as well, though at the level of the state space of the C*-algebra: namely the subsets $\StateSpace(\A|K)$, where $K$ ranges among all compact subsets of the Gel'fand spectrum of $\M$ for $(\A,\M)$ a topographic quantum space. While quite different in practice, these notions both attempt in their own way to propose a meaning for quantum locality. 

A key difference is that our notion of {\qms} relies on the topology on the state space induced by the {\mongekant} whereas the notion of a $W^\ast$-metric does not seem to have any such topological requirements. In \cite[Definition 4.19]{Kuperberg10}, a Lipschitz seminorm is defined from a $W^\ast$-metric, and it would be interesting to investigate, for the various examples introduced in \cite{Kuperberg10}, when such Lipschitz seminorms can be used to define {\qms s}. One would have, for a given $W^\ast$-metric on some Von Neumann algebra, to consider the closure of the Lipschitz algebras for the Von Neumann algebra norm and find a proper topography for these structures --- presumably, the topography should be related with the filtration. The topological questions which we raise here would require much work to be addressed in this fascinating setting, and we leave this endeavor for further publications.

We take a small detour from our main purpose to characterize Lipschitz triples which are indeed {\qms s} in order to establish a property which answers a natural question: given a {\qms} $(\A,\Lip,\M)$, what geometry does the spectrum of $\M$ inherits from the metric data? This is a bit subtle and the subject of the following result.

\begin{theorem}\label{qms-topographic-metric-thm}
Let $(\A,\Lip,\M)$ be a {\qms}. Let $\sigma(\M)$ be the Gel'fand spectrum of $\M$. For any two states $\omega,\rho$ of $\M$, define:
\begin{equation*}
\mathsf{d}(\omega,\rho) = \inf \{ \Kantorovich{\Lip}(\varphi,\psi) : \varphi,\psi \in \StateSpace(\A)\text{ and } [\varphi]_\M = \omega\text{ and } [\psi]_\M = \rho \} \text{.}
\end{equation*}
Then $\mathsf{d}$ is an extended metric on $\StateSpace(\M)$ which, for all $K\in\compacts{\sigma(\M)}$, metrizes the relative topology induced by the weak* topology $\sigma(\M^\ast,\M)$ on $\{ \varphi \in \StateSpace(\M): \varphi(\indicator{K})=1\}$. In particular, the restriction of $\mathsf{d}$ to $\sigma(\M)$ metrizes the topology of $\sigma(\M)$.
\end{theorem}

\begin{proof}
We first prove that $\mathsf{d}$ thus defined is an extended metric on $\StateSpace(\M)$. First, assume $\mathsf{d} (\rho, \omega ) = 0$. Let $\varepsilon > 0$. By definition of $\mathsf{d}$, there exists $\varphi_\varepsilon, \psi_\varepsilon \in \StateSpace(\A)$ such that $[\varphi_\varepsilon]_\M = \rho$ and $[\psi_\varepsilon]_\M = \omega$, while $\Kantorovich{\Lip}(\varphi_\varepsilon,\psi_\varepsilon) < \varepsilon$.  Recall from \cite{Pedersen79} that the quasi-state space of $\A$ is the set of all linear positive functionals of $\A$ with norm at most $1$, and this set is weak* compact. Since the quasi-state space of $\A$ is weak* compact, there exists some directed set $I$ and some cofinal monotone function $\alpha : I \rightarrow (0,\infty)$ (where the order on the directed set $(0,\infty)$ is dual to the usual order) such that the subnet $(\varphi_{\alpha(\lambda)})_{\lambda\in I}$ of $(\varphi_\varepsilon)_{\varepsilon>0}$ is  weak* convergent to some quasi-state $\varphi$. We then have $[\varphi]_\M = \rho$ and thus $\varphi \in \StateSpace(\A)$. By weak* compactness of the quasi-state-space again, there exists a directed set $J$ and a cofinal, monotone map $\beta : J \rightarrow I$ such that the subnet $(\psi_{\alpha\circ\beta(\lambda)})_{\lambda\in J}$ of $(\psi_{\alpha(\lambda)})_{\lambda\in I}$ weak* converges to some quasi-state $\psi$. Since $[\psi]_\M = \omega$, we have $\psi \in \StateSpace(\A)$. Thus, $(\varphi_{\alpha\circ\beta(\lambda)})_{\lambda\in J}$ and $(\psi_{\alpha\circ\beta(\lambda)})_{\lambda\in J}$ are weak* convergent to, respectively, the states $\varphi$ and $\psi$.

Fix $a\in\sa{\A}$ with $\Lip(a)\leq 1$ and $\varepsilon > 0$. By construction, there exists $\lambda_0 \in J$ such that for all $\lambda \in J$ with $\lambda \succ \lambda_0$ we have $\alpha\circ\beta(\lambda) < \frac{1}{3}\varepsilon$. Thus by construction:
\begin{equation*}
\forall \lambda\in J \quad (\lambda\succ \lambda_0) \implies \Kantorovich{\Lip}(\varphi_{\alpha\circ\beta(\lambda)},\psi_{\alpha\circ\beta(\lambda)}) < \frac{1}{3}\varepsilon\text{.}
\end{equation*} 
Then, by weak* convergence, there exists $\lambda_1,\lambda_2 \in J$ such that:
\begin{equation*}
\forall \lambda\in J\quad (\lambda \succ \lambda_1)\implies|\varphi_{\alpha\circ\beta(\lambda)}(a) - \varphi(a)| < \frac{1}{3}\varepsilon\text{,}
\end{equation*}
and:
\begin{equation*}
\forall \lambda\in J\quad (\lambda\succ \lambda_2)\implies |\psi_{\alpha\circ\beta(\lambda)}(a) - \psi(a)| < \frac{1}{3}\varepsilon\text{.}
\end{equation*}
Since $J$ is a directed set, there exists $\lambda \in J$ with $\lambda\succ \lambda_0, \lambda\succ \lambda_1, \lambda \succ \lambda_2$. Thus:
\begin{equation*}
\begin{split}
|\varphi(a)-\psi(a)| &\leq |\varphi(a)-\varphi_{\alpha\circ(\lambda)}(a)| + |\varphi_{\alpha\circ\beta(\lambda)}(a) - \psi_{\alpha\circ\beta(\lambda)}(a)| \\
&\quad + |\psi_{\alpha\circ\beta(\lambda)}(a) - \psi(a)| \\
&< \varepsilon \text{.}
\end{split}
\end{equation*}
Hence, as $\varepsilon > 0$ was arbitrary, we have $\varphi(a)=\psi(a)$.

Therefore, $\Kantorovich{\Lip}(\varphi,\psi) = 0$. Now, $\Kantorovich{\Lip}$ is an extended metric, so $\varphi = \psi$ and thus $\rho = \omega$.

Symmetry is clear. The triangle inequality requires a bit of notation, and can be established as follows.

Let:
\begin{equation*}
\mathsf{l} : \psi \in \A^\ast \mapsto \sup \{ |\psi(a)| : a\in\sa{\A}\text{ and } \Lip(a)\leq 1\}
\end{equation*}
and note that $\Kantorovich{\Lip}(\varphi,\psi) = \mathsf{l}(\varphi-\psi)$. Of course, $\mathsf{l}$ may take the value $\infty$, but it is a seminorm on the subset $\{\lambda \in \A^\ast : \mathsf{l}(\lambda)<\infty\}$. Now, let:
\begin{equation*}
\mathsf{l}' : \omega \in \M^\ast \mapsto \inf \{ \mathsf{l}(\psi) : [\psi]_\M = \omega \}\text{.}
\end{equation*}
Note that by construction, $\mathsf{d}(\omega,\rho) = \mathsf{l}'(\omega-\rho)$. On the other hand, $\mathsf{l}'$ is the quotient seminorm of $\mathsf{l}$ by the space $\{\psi \in \A^\ast : [\psi]_\M = 0 \}$. Thus $\mathsf{d}$ satisfies the triangle inequality.
 
Let $C\in \compacts{\sigma(\M)}$ be given. Let $S_C$ be the space of all Radon probability measures on $\sigma(\M)$ supported in $C$. By definition, $[\StateSpace(\A|C)]_\M = S_C$. Since $(\A,\Lip,\M)$ is a {\qms}, and since $\StateSpace(\A|C)$ is a tame subset of $\StateSpace(\A)$, the weak* topology of $\StateSpace(\A|C)$ is metrized by $\Kantorovich{\Lip}$. In particular, since $\StateSpace(\A|C)$ is weak*-compact and metrizable, it is weak*-separable. Hence $S_C$ is weak* separable. Assume $(\varphi_n)_{n \in \N}$ is a sequence in $S_C$ weak* converging to some $\varphi_\infty$. Since $S_C$ is weak* compact, $\varphi_\infty \in S_C$. Now, let $(\varphi_{m(n)})_{n\in\N}$ be an arbitrary subsequence of $(\varphi_n)_{n \in \N}$. Let $(\psi_n)_{n\in\N}$ be a sequence in $\StateSpace(\A|C)$ such that for all $n\in\N$ we have $[\psi_n]_\M = \varphi_{m(n)}$. Since $\StateSpace(\A|C)$ is weak* compact, there exists a subsequence $(\psi_{s(n)})_{n\in\N}$ which weak* converges to some $\psi_\infty \in \StateSpace(\A|C)$. By weak* continuity of the restriction map and uniqueness of the weak* limit, we conclude that $[\psi_\infty] = \varphi_\infty$.
On the other hand, since $\Kantorovich{\Lip}$ metrizes the weak* topology on $\StateSpace(\A|C)$ because $(\A,\Lip,\M)$ is a {\qms}, we conclude that $\lim_{n\rightarrow \infty}\Kantorovich{\Lip}(\psi_{s(n)},\psi_\infty) = 0$. Hence by definition, $\lim_{n\rightarrow\infty}\mathsf{d}(\varphi_{m(s(n))},\varphi_\infty) = 0$.

Thus, the sequence $(\mathsf{d}(\varphi_n,\varphi_\infty))_{n\in\N}$ has the property that every subsequence has a subsequence converging to $0$. Hence, the sequence itself converges to $0$, as desired.

Identifying $\mathsf{d}$ with the distance it induces on $C$ by identifying points of $C$ with their Dirac probability measures, we see that $C$ is metrized by $\mathsf{d}$. This is sufficient to conclude that $\mathsf{d}$ metrizes $\sigma(\M)$: if $(x_\alpha)_{\alpha\in I}$ is some net in $\sigma(\M)$ which converges to some $x\in \sigma(\M)$ then the set $\{x_\alpha,x : \alpha \in I\}$ is compact in $\sigma(\M)$, and thus metrized by $\mathsf{d}$.
\end{proof}

Theorem (\ref{qms-topographic-metric-thm}) gives a necessary regularity condition of the topography of a {\qms}, which brings potentially useful topological results in our context. We record this necessary condition here.

\begin{corollary}
Let $(\A,\Lip,\M)$ be a {\qms}. Then $\sigma(\M)$ is a paracompact locally compact metric space.
\end{corollary}

\begin{proof}
This follows from a Theorem of A. H. Stone which states that all metric spaces are paracompact in \cite{Stone48}.
\end{proof}

Since $\sigma(\M)$ is a metric space, it is natural to endow it with the associated Lipschitz seminorm. One should note that this seminorm is usually dominated by, and not equal to $\Lip$.

While this paper lays the foundation for our work in the locally compact context for metric noncommutative geometry, we find it a good place to introduce a notion which will occupy the central role in \cite{Latremoliere12c} where we develop the quantum Gromov-Hausdorff convergence for pointed {\qmss}:

\begin{definition}\label{p-qms-def}
A \emph{pointed \qms} $(\A,\Lip,\M,\mu)$ is a {\qms} $(\A,\Lip,\M)$ and a state $\mu \in \StateSpace(\A|\M)$ such that $[\mu]_\M$ is pure. 
\end{definition}

Definition ({\ref{qms-def}) involves proving that the {\mongekant} gives the weak* topology on all tame subsets of the local state space of a regular Lipschitz triple. This should appear quite a daunting task in general: one may find it rather challenging to describe all tame sets beyond the very definition we have given for them. Of course, it should not be too surprising that constructing {\qms s} is a arduous task, as it already is so with compact quantum metric spaces. On the other hand, a characterization of {\qms} more amenable to C*-algebraic methods would prove quite useful. Rieffel's insight for compact quantum metric space was to characterize compact quantum metric spaces in terms of the unit ball for the Lip-norm.

We propose a similar characterization, though of course ours will be a bit more involved to handle our greater level of generality. A strong hint for our characterization is our work in \cite{Latremoliere05b}, whose main result was recalled in our introduction. Our characterization will make minimal reference to the state space of a {\qms} and rather focus on a C*-algebraic criterion of total boundedness of a family of sets for the norm. To obtain such results, the bridge on which information passes from the state space to the C*-algebra is a topology on the latter defined through the former. This is the matter of the next section.

\subsection{The Main Characterization}

The weakly uniform topology of \cite{Latremoliere05b} was well suited for bounded quantum metric spaces, where all weak* compact sets are tame, as we shall discuss in the Examples section later in this paper. Yet, in general, our tameness condition depends on the topographic substructure, and thus our new topology does too. We introduce:

\begin{definition}
Let $(\A,\M)$ be a topographic quantum space. We define for any $K\in\compacts{\sigma(\M)}$ the seminorm: 
\begin{equation*}
\pnorm{K} : a \in \unital\A \mapsto \|\indicator{K}a\indicator{K}\|_{\A^{\ast\ast}}\textrm{.}
\end{equation*}
The $(\A,\M)$-{\tu} is the locally convex topology on $\unital{\A}$ generated by the set of seminorms $\{ \pnorm{K} : K \in \compacts{\sigma(\M)} \}$. We denote this topology by {\tut$(\A,\M)$}.
\end{definition}

\begin{proposition}
Let $(\A,\M)$ be a topographic quantum space. The $(\A,\M)$-{\tu} is Hausdorff.
\end{proposition}

\begin{proof}
Let $a\in \A$ such that for all $K \in \compacts{\sigma(\M)}$, we have $\pnorm{K}(a) = 0$. Thus $\|\indicator{K}a\indicator{K}\|_{\A^{\ast\ast}} = 0$. If $f_1,f_2 \in \sa{\M}$ are compactly supported, with respective supports $K_1, K_2$, then:
\begin{equation*}
\begin{split}
\|f_1af_2\|_\A &= \|f_1\indicator{(K_1\cup K_2)} a \indicator{(K_1\cup K_2)}f_2\|_{\A^{\ast\ast}} \\
&\leq \|f_1\|_\A\|f_2\|_\A\|\indicator{(K_1\cup K_2)}a\indicator{(K_1\cup K_2)}\|_{\A^{\ast\ast}} = 0\text{.}
\end{split}
\end{equation*}
Hence, for the approximate identity $(f_\alpha)_{\alpha\in I}$ given by Lemma (\ref{topographic-approx-identity-lem}) and all $\alpha\in I$:
\begin{equation*}
\begin{split}
0\leq \|a\|_\A &\leq \|a-f_\alpha a f_\alpha\|_\A + \|f_\alpha a f_\alpha\|_\A = \|a-f_\alpha a f_\alpha\|_\A \\
&\leq \|a-f_\alpha a\|_\A + \|f_\alpha\|_\A\|a-af_\alpha\|_{\A} \stackrel{\alpha\in I}{\longrightarrow} 0\text{,}
\end{split}
\end{equation*}
i.e. $a=0$ as desired.
\end{proof}

The following comparison between our topographic uniform topology and our weakly uniform topology, from \cite{Latremoliere05b}, places {\tut} among the many classical topologies on a C*-algebra. For the convenience of the reader, we shall now recall the definition of the weakly uniform topology, and refer to \cite{Latremoliere05b} for its relation to the weak, strict, strongly uniform and norm topologies.

\begin{definition}
Let $\A$ be a C*-algebra and let $\mathfrak{S}$ be the set of all weak* compact subsets of $\StateSpace(\A)$. For each $\mathscr{K} \in \mathfrak{S}$ we define:
\begin{equation*}
p_{\mathscr{K}} : a \in \A \longmapsto \sup \{ |\varphi(a)| : \varphi \in \mathscr{K} \} \text{.}
\end{equation*}
Then $p_{\mathscr{K}}$ is a seminorm on $\A$ for all $\mathscr{K} \in \mathfrak{S}$. The \emph{weakly uniform topology} on $\A$ is the locally convex topology on $\A$ generated by the set $\{ p_{\mathscr{K}} : \mathscr{K}\in\mathfrak{S} \}$ of seminorms of $\A$.
\end{definition}

\begin{proposition}\label{bounded-wu-eq-tu-prop}
Let $(\A,\M)$ be a topographic quantum space and $\B\subseteq \A$ be a bounded set. Then the {\tu} and the weakly uniform topology gives the same relative topology to $\B$.
\end{proposition}

\begin{proof}
Let $M = \sup \{ \|a\|_\A : a\in\mathfrak{B}\}$. The weakly uniform topology is stronger than {\tut} by definition. Let $(a_\alpha)_{\alpha \in I}$ be a net in $\B$ converging in $\B$ to $a$ for the {\tu}. Let $\varepsilon \in (0,1)$. Let $\mathscr{K}$ be a weak* compact subset of $\StateSpace(\A)$. By Theorem (\ref{tight-eq-precompact-thm}), $\mathscr{K}$ is tight, so there exists $K\in\compacts{\sigma(\M)}$ such that for all $C\in\compacts{\sigma(\M)}$ with $K\subseteq C$ we have $\sup \{ \varphi(\unit_{\unital{\A}}-\indicator{C}) : \varphi \in \mathscr{K} \} \leq \varepsilon$. On the other hand, for all $\varphi \in \mathscr{K}$, Inequation (\ref{local-state-space-dense-prop-eq0}) in the proof of Proposition (\ref{local-state-space-dense-prop}) gives us, for all $a\in \A$ and $\varphi \in \mathscr{K}$:
\begin{equation*}
|\varphi(a)-\psi_K(a)| \leq  2\varepsilon |\varphi(a)| + 2\varphi(1-\indicator{K})\|a\|_\A \leq 4M\varepsilon \text{,}
\end{equation*}
where $\psi_K : a\in\A\mapsto \varphi(\indicator{K})^{-1}\varphi(\corner{K}{a}) \in \StateSpace(\A|K)$.

Now, by assumption, $(a_\alpha)_{\alpha\in I}$ converges uniformly to $a$ for $\pnorm{K}$. Let $\beta\in I$ such that for all $\beta\succ\alpha$ and for all $\psi \in \StateSpace(\A|K)$ we have $|\psi(a)-\psi(a_\beta)|\leq\varepsilon$. Then for all $\varphi \in \mathscr{K}$ we have $|\varphi(a_\beta)-\varphi(a)| \leq (8M+1)\varepsilon$. Hence $\lim_{\alpha\in I}\sup \{|\varphi(a_\alpha)-\varphi(a)| : \varphi \in \mathscr{K} \} = 0$ as desired.
\end{proof}

We shall see in the Examples sections that, in general, the {\tu} is strictly weaker than the weakly uniform topology.

Now, by definition, a subset $\B$ of $\A$ is totally bounded for the seminorm $\pnorm{K}$ if and only if the set $\indicator{K}\B\indicator{K}$ is totally bounded in the norm of $\A^{\ast\ast}$. With this observation in mind, the central theorem of this paper is the following bridge result, of which our two other characterizations are corollaries.

\begin{theorem}\label{qms-characterization-via-lu-topology-thm}
Let $(\A,\Lip,\M)$ be a Lipschitz triple. The following assertions are equivalent:
\begin{enumerate}
\item $(\A,\Lip,\M)$ is a {\qms},
\item For all $\mu \in \StateSpace(\A|\M)$, the set:
\begin{equation*}
\mulip(\A,\Lip,\mu) = \{ a \in \unital{\A} : \Lip(a)\leq 1 \text{ and } \mu(a) = 0 \}
\end{equation*}
is totally bounded in the $(\A,\M)$-{\tu}, i.e. for all $K \in \compacts{\sigma(\M)}$, the set $\indicator{K}\mulip(\A,\Lip,\mu)\indicator{K}$ is norm totally bounded in $\A^{\ast\ast}$,
\item For some $\mu \in \StateSpace(\A|\M)$ and for all $K \in \compacts{\sigma(\M)}$, the set $\indicator{K}\mulip(\A,\Lip,\mu)\indicator{K}$ is norm totally bounded in $\A^{\ast\ast}$ (i.e. $\mulip(\A,\Lip,\mu)$ is $(\A,\M)$-{\tu}-totally bounded).
\end{enumerate}
\end{theorem}

\begin{proof}
Assume that $(\A,\Lip,\M)$ is a {\qms}. Let $\mu$ be a local state. Let $K\in\compacts{\sigma(\M)}$. For any $a\in \sa{\unital{\A}}$ we define:
\begin{equation*}
\Theta_{\StateSpace(\A|K)}(a) : \varphi \in \StateSpace(\A|K) \longmapsto \varphi(a) \text{.}
\end{equation*}
The map $\Theta_{\StateSpace(\A|K)}(a)$ is continuous from the weak* topology on $\StateSpace(\A|K)$ to $\R$ for all $a\in\sa{\unital{\A}}$. Let $C(\StateSpace(\A|K))$ be the real Banach space of real valued weak* continuous functions on $\StateSpace(\A|K)$, with the supremum norm denoted by $\|\cdot\|_{C(\StateSpace(\A|K))}$.

Now, since $(\A,\Lip,\M)$ is a regular Lipschitz triple, the {\mongekant} $\Kantorovich{\Lip}$ restricted to $\StateSpace(\A|K)$ is in fact a metric, and there exists $r\in \R$ such that $\Kantorovich{\Lip}(\varphi,\mu) \leq r$ for all $\varphi \in \mathscr{K}$ by regularity. Since $\mu$ is local, there exists $K'\in\compacts{\sigma(\M)}$ with $K\subseteq K'$ and $\mu\in\StateSpace(\A|K')$, the latter being bounded for $\Kantorovich{\Lip}$. So, for all $a\in \mulip(\A,\Lip,\mu)$ we have:
\begin{align*}
\|\Theta_{\StateSpace(\A|K)}(a)\|_{C(\StateSpace(\A|K))} &= \sup \{ |\varphi(a)| : \varphi \in \StateSpace(\A|K) \}\\
&=  \sup \{ |\varphi(a) - \mu(a)| : \varphi \in \StateSpace(\A|K) \}\\
&\leq r \text{.}
\end{align*}
Since $(\A,\Lip,\M)$ is a {\qms} and $\StateSpace(\A|K)$ is tame by Proposition (\ref{obvious-tame-prop}), the topology induced by $\Kantorovich{\Lip}$ on $\StateSpace(\A|K)$ is the relative topology induced by the weak* topology. On the other hand, by definition of $\Kantorovich{\Lip}$, we have, for all $a\in\mulip(\A,\Lip,\mu)$, that:
\begin{equation*}
|\Theta_{\StateSpace(\A|K)}(a)(\varphi) - \Theta_{\StateSpace(\A|K)}(a)(\psi)| = |\varphi(a)-\psi(a)|\leq\Kantorovich{\Lip}(\varphi,\psi)
\end{equation*}
so $\Theta_{\StateSpace(\A|K)}(a)$ is $1$-Lipschitz on $(\StateSpace(\A|K),\Kantorovich{\Lip})$. Thus, the set:
\begin{equation*}
\Theta_{\StateSpace(\A|K)}(\mulip(\A,\Lip,\M)) = \{ \Theta_{\StateSpace(\A|K)}(a) : a\in\mulip(\A,\Lip,\mu) \}
\end{equation*}
is an equicontinuous, bounded subset of $C(\StateSpace(\A|K))$, and $\StateSpace(\A|K)$ is weak* compact, so by the Arz{\'e}la-Ascoli theorem, $\Theta_{\StateSpace(\A|K)}(\mulip(\A,\Lip,\M))$ is totally bounded for $\|\cdot\|_{C(\StateSpace(\A|K))}$. Now, by definition and Lemma (\ref{pnorm-eq-norm-lem}), we have:
\begin{equation*}
\|\Theta_{\StateSpace(\A|K)}(a)\|_{C(\StateSpace(\A|K))} = \pnorm{K}(a)
\end{equation*}
for all $a\in\unital\A$, so we conclude that $\mulip(\A,\Lip,\M)$ is totally bounded for $\pnorm{K}$ as desired. So $\mulip(\A,\Lip,\mu)$ is {\tut}-totally bounded.

Thus (1) implies (2). Of course, (2) implies (3) trivially. 

Last, assume that $\mulip(\A,\Lip,\mu)$ is {\tut}-totally bounded for some local state $\mu$. First, by Proposition (\ref{regular-characterization-prop}), since for all $K\in\compacts{\sigma(\M)}$ the set $\indicator{K}\mulip(\A,\Lip,\M)\indicator{K}$ is totally bounded in $\A^{\ast\ast}$, hence bounded, in norm, the Lipschitz triple $(\A,\Lip,\M)$ is a regular.

Let $\mathscr{K}$ be a $(\A,\Lip,\M)$-tame subset of $\StateSpace(\A)$.  By Proposition (\ref{wasskant-finer-than-weak-prop}), the {\mongekant} $\Kantorovich{\Lip}$ induces a finer topology on $\mathscr{K}$ than the weak* topology. It is thus sufficient to show that the weak* topology is finer than the metric topology induced by $\Kantorovich{\Lip}$.

Let $(\varphi_\alpha)_{\alpha\in I}$ be a net in $\mathscr{K}$ indexed by a directed set $I$, and weak* convergent to $\varphi$ in $\mathscr{K}$. The set $\Phi = \{ \varphi_\alpha,\varphi : \alpha \in I\}$ is weak* compact by construction and tame, as it is a subset of a tame set.

We start with the key role than tameness plays in this proof. Since $\Phi$ is tame, there exists $K \in \compacts{\sigma(\M)}$ such that:
\begin{equation}
\sup \{ |\psi( a - \indicator{K}a\indicator{K} )| : \psi \in \Phi, a \in \mulip(\A,\Lip,\mu) \} \leq \frac{1}{9}\varepsilon\text{.}
\end{equation}

Since $\mulip(\A,\Lip,\M)$ is totally bounded for $\pnorm{K}$, there exists a finite subset $F\subseteq \mulip(\A,\Lip,\mu)$ such that:
\begin{equation*}
 \forall a\in\mulip(\A,\Lip,\mu) \quad \exists f(a) \in F\quad\quad\pnorm{K}(f(a)-a)\leq\frac{1}{9}\varepsilon\text{.}
\end{equation*}

We then have by definition that:
\begin{equation*}
|\psi(\indicator{K}a\indicator{K})-\psi(\indicator{K} f(a) \indicator{K})| \leq \pnorm{K}(a-f(a))
\end{equation*}
for all $\psi\in\StateSpace(\A)$ and $a\in \sa{\unital{\A}}$. Thus for all $\psi \in \Phi$ and $a\in\mulip(\A,\Lip,\M)$:
\begin{equation}
\begin{split}
|\psi(a) - \psi(f(a))| &\leq |\psi(a)-\psi(\indicator{K}a\indicator{K})| + |\psi(\indicator{K}a\indicator{K})-\psi(\indicator{K} f(a) \indicator{K})|\\ &\quad + |\psi(\indicator{K}f(a)\indicator{K} - f(a))| \leq \frac{1}{3}\varepsilon \text{.}
\end{split}
\end{equation}

Let $\omega \in I$ be chosen so that if $\alpha\succ \omega$ and $b\in F$ then $|\varphi_\alpha(b) - \varphi(b)|\leq \frac{1}{3}\varepsilon$. This is of course possible since $I$ is directed and $F$ is finite. 

Now, for any $a\in\mulip(\A,\Lip,\mu)$ and $\alpha\in I$ with $\alpha\succ\omega$, we have:
\begin{align*}
|\varphi_\alpha(a)-\varphi(a)| &\leq |\varphi_\alpha(a) - \varphi_\alpha(f(a))| + |\varphi_\alpha(f(a)) - \varphi(f(a)) | + |\varphi(f(a))-\varphi(a)| \\
&\leq \frac{1}{3}\varepsilon + \frac{1}{3}\varepsilon + \frac{1}{3}\varepsilon = \varepsilon \text{.}
\end{align*}
Hence $\Kantorovich{\Lip}(\varphi_\alpha,\varphi)\leq \varepsilon$ for $\alpha\succ\omega$, thus proving that the net $(\varphi_\alpha)_{\alpha \in I}$ converges to $\varphi$ in for $\Kantorovich{\Lip}$ as desired.
\end{proof}

\subsection{C*-algebraic characterizations of Topographic Quantum Locally Compact Metric Spaces}

We offer two alternative characterizations for {\qms s}. 

This section first offers a characterization of {\qms s} which avoids the use of projections, which involve working in the enveloping Von Neumann algebra of a C*-algebra and thus may in general be challenging. The second result is a characterization for {\qmss s}. Many examples in practice will be {\qmss s}, and the addition of the separability condition leads to a very nice characterization.

For an Abelian C*-algebra $\M$, the notion of a compactly supported element is well-defined, by identifying $\M$ with the C*-algebra $C_0(\sigma(\M))$ of continuous functions on the spectrum of $\M$ vanishing at infinity. Our main result for this section is:

\begin{theorem}\label{qms-characterization-cc-thm}
Let $(\A,\Lip,\M)$ be a Lipschitz triple. The following assertions are equivalent:
\begin{enumerate}
\item $(A,\Lip,\M)$ is a {\qms},
\item For all $\mu \in \StateSpace(\A|\M)$ and for all \emph{compactly supported} $a,b \in \M$, the set $a \mulip(\A,\Lip,\mu) b$ is totally bounded in the norm topology of $\unital{\A}$,
\item  For some $\mu \in \StateSpace(\A|\M)$ and for all \emph{compactly supported} $a,b \in \M$, the set $a \mulip(\A,\Lip,\mu) b$ is totally bounded in the norm topology of $\unital{\A}$.
\end{enumerate}
\end{theorem}   

\begin{proof}
Assume first that $(\A,\Lip,\M)$ is a {\qms}. By Theorem (\ref{qms-characterization-via-lu-topology-thm}),  $\indicator{K}\mulip(\A,\Lip,\mu) \indicator{K}$ is totally bounded in norm for all $K \in \compacts{\sigma(\M)}$ and all local states $\mu$. Let $a,b \in \M$ with compact support. Then there exists $K \in \compacts{\sigma(\M)}$ such that $a \indicator{K} = a$ and $\indicator{K} b = b$. Hence for all $c \in \unital{\A}$:
\begin{equation*}
\|acb\|_\A \leq \|a\indicator{K}c\indicator{K}b\|_{\A^{\ast\ast}} \leq \|a\|_\A\|b\|_\B \|\indicator{K}c\indicator{K}\|_{\A^{\ast\ast}} \text{.}
\end{equation*}
Hence, for all local states $\mu$, since $\indicator{K}\mulip(\A,\Lip,\mu)\indicator{K}$ is totally bounded in $\|\cdot\|_{\A^{\ast\ast}}$, so is $a \mulip(\A,\Lip,\mu) b$. As $a \mulip(\A,\Lip,\mu) b \subseteq \unital{\A}$, and $\|\cdot\|_{\unital{\A}}$ equals to the restriction of $\|\cdot\|_{\A^{\ast\ast}}$ to $\unital{\A}$, we have proven that (1) implies (2).

The second assertion obviously implies the third.

Assume now that for some $\mu \in \StateSpace(\A|\M)$ and for all \emph{compactly supported} $a,b \in \M$, the set $a \mulip(\A,\Lip,\mu) b$ is totally bounded in the norm topology of $\unital{\A}$.

Let $K\in\compacts{\sigma(\M)}$. There exists $c \in \sa{\M}$ compactly supported such that $\indicator{K}c = \indicator{K}$ by Urysohn's lemma for locally compact Hausdorff spaces. Hence, for all $a\in\mulip(\A,\Lip,\mu)$ we have:
\begin{equation*}
\|\indicator{K}a\indicator{K}\|_{\A^{\ast\ast}} = \|\indicator{K}cac\indicator{K}\|_{\A^{\ast\ast}} \leq \|cac\|_{\unital{\A}} \text{.}
\end{equation*}
We thus easily deduce that, since $c\mulip(\A,\Lip,\mu) c$ is totally bounded for $\|\cdot\|_{\unital{\A}}$, the set $\indicator{K}\mulip(\A,\Lip,\mu)\indicator{K}$ is totally bounded for $\|\cdot\|_{\A^{\ast\ast}}$. By Theorem (\ref{qms-characterization-via-lu-topology-thm}), we conclude that $(\A,\Lip,\M)$ is a {\qms}. This proves our theorem.
\end{proof}

We now turn to the important special case of {\qmss s}. We note that by assumption, if $(\A,\M)$ is a topographic quantum space and $h\in \M$ is a strictly positive element in $\M$ then it is also a strictly positive element in $\A$, and conversely if $h\in\M$ is strictly positive for $\A$ then it is so as well in $\M$. This follows from Proposition (\ref{restriction-prop}), for instance.

\begin{theorem}\label{qms-characterization-separable-thm}
Let $(\A,\Lip,\M)$ be a Lipschitz triple where $\A$ is separable. The following assertions are equivalent:
\begin{enumerate}
\item  $(\A,\Lip,\M)$ is a {\qmss},
\item There exists a strictly positive element $h \in \M$ such that for all $\mu \in \StateSpace(\A|\M)$, the set $h\mulip(\A,\Lip,\mu) h$ is totally bounded for $\|\cdot\|_{\unital{\A}}$,
\item There exists a strictly positive element $h \in \M$ and a local state $\mu$ such that $h\mulip(\A,\Lip,\mu) h$ is totally bounded for $\|\cdot\|_{\unital{\A}}$.
\end{enumerate}
\end{theorem}

\begin{proof}
Assume that $(\A,\Lip,\M)$ is a {\qmss}. Fix a local state $\mu$.

Since $\A$ is separable, so is $\M$, so $\sigma(\M)$ is a separable locally compact Hausdorff space. Let $(K_n)_{n\in\N}$ be an increasing sequence of compact subsets of $\sigma(\M)$ such that $\bigcup_{n\in\N} K_n = \sigma(\M)$ and $K_n \subseteq (\operatorname{int} K_{n+1})$ for all $n\in\N$ where $(\operatorname{int} T)$ is the topological interior of any subset $T$ of $\sigma(\M)$. For instance, since there exists $f \in \M$ strictly positive as $\M$ separable \cite{Pedersen79}, we could choose $K_n = f^{-1} \left(\left[\frac{1}{n+1},\|f\|_\M\right]\right) $ for all $n\in \N$. Since $f\in C_0(\sigma(\M))$ and $f \geq 0$, the set $K_n$ is compact for all $n\in\N$.  We also note that $K_n \subseteq f^{-1}\left(\left(\frac{1}{n+2},\|f\|_\M\right]\right)\subseteq \operatorname{int}K_{n+1} $ for all $n\in\N$. Last, if $x\in \sigma(\M)$ then $f(x) > \frac{1}{n+1}$ for some $n\in\N$ since $f$ is strictly positive, so $x\in K_{n}$.

Let $p_n = \indicator{K_n}$ and $c_n\in \sa{\M}$ be such that $c_np_n = p_n$ and $c_n (1-p_m) = 0$ while $\|c_n\|_{\unital{\A}} = 1$ for all $n,m\in\N$ with $m>n$. Such a construction is done simply by induction and using the Urysohn's Lemma for locally compact spaces. Note that by construction, $c_n$ is compactly supported in $\M$ for all $n\in\N$

We now select a sequence $(x_n)_{n\in\N}$ of nonnegative real numbers as follows. For $n\in \N$, the set $c_n \mulip(\A,\Lip,\mu) c_n$ is totally bounded in $\|\cdot\|_{\unital{\A}}$ by Theorem (\ref{qms-characterization-cc-thm}), hence bounded by some $R_n \in \R$. Let $x_n = 2^{-n} (\max\{R_n, 1\})^{-1}$. 

Since $(\sum x_n c_n)_{n\in\N}$ is absolutely summable by construction, we can define \begin{equation*}
h = \sum_{n=0}^\infty x_n c_n \in \M\text{.}
\end{equation*} 
It is easy to check that by construction, $(c_n)_{n\in\N}$ is an approximate unit of $\M$, so the element $h$ is strictly positive in $\M$, and hence in $\A$ by Proposition (\ref{restriction-prop}).

Now, since $(c_n)_{n\in\N}$ satisfies the relations $c_nc_m = c_n$ if $n\leq m$ for all $n,m\in\N$, we note that for all $a\in\unital{\A},n,m\in\N$, if $n\leq m$ then:
\begin{equation*}
\|c_n a c_m\|_{\unital{\A}} = \|c_n c_m a c_m\|_{\unital{\A}} \leq \|c_m a c_m\|_{\unital{\A}}
\end{equation*}
and similarly $\|c_m a c_n\|_{\unital{\A}}\leq \|c_ma c_m\|_{\unital{\A}}$. Also, by construction, we have $\sum_{n=0}^\infty x_n \leq 1$. Hence, for all $a \in \mulip(\A,\Lip,\mu)$, we have:
\begin{align}
\|hah\|_{\unital{\A}} &\leq \sum_{n=0}^\infty\sum_{m=0}^\infty x_nx_m \|c_n a c_m\|_{\unital{\A}} \\
&\leq \sum_{n=0}^\infty (\sum_{m=0}^n x_m) x_n \|c_n a c_n\|_{\unital{\A}} + \sum_{n=0}^\infty \sum_{m=n+1}^\infty x_m x_n \|c_m a c_m\|_{\unital{\A}}\\
&\leq \sum_{n=0}^\infty x_n \|c_n a c_n\|_{\unital{\A}} + \sum_{m=1}^\infty (\sum_{n=0}^m x_n) x_m \|c_m a c_m\|_{\unital{\A}}\\
&\leq 2 \sum_{n=0}^\infty x_n\|c_n a c_n\|_{\unital{\A}} \text{.}
\end{align}

Let $\varepsilon > 0$. Let $N\in \N$ such that $\sum_{n=N+1}^\infty 2^{-n} \leq \frac{1}{2} \varepsilon$. By the choice of $(x_n)_{n\in\N}$, for all $a\in \mulip(\A,\Lip,\mu)$ we have: 
\begin{equation*}
\sum_{n=N+1}^\infty x_n \|c_n a c_n\|_{\unital{\A}} \leq \sum_{n=N+1}^\infty 2^{-n} \leq \frac{1}{2} \varepsilon\text{.}
\end{equation*} 

On the other hand, since $c_Nc_n = c_n$ for all $n\leq N$, if we set $M = \sum_{n=0}^N x_n$ then we have for all $a\in \unital{\A}$:
\begin{equation*}
\sum_{n=0}^N x_n \|c_n a c_n\|_{\unital{\A}} \leq M \|c_N a c_N \|_{\unital{\A}}\text{.}
\end{equation*}

Since $(\A,\Lip,\M)$ is a {\qms}, by Theorem (\ref{qms-characterization-cc-thm}). the set $c_N \mulip(\A,\Lip,\mu) c_N$ is totally bounded for $\|\cdot\|_{\unital{\A}}$. Therefore, there exists a finite subset $F$ of $\mulip(\A,\Lip,\mu)$ such that for all $a\in \mulip(\A,\Lip,\M)$ there exists $f(a)\in\mulip(\A,\Lip,\mu)$ with $\|c_N (a-f(a))c_N\|_{\unital{\A}} \leq \frac{1}{2M} \varepsilon$.

Thus:
\begin{equation*}
\|h(a-f(a))h\|_{\unital{\A}} \leq \frac{1}{2}\varepsilon + M\|c_N(a-f(a))c_N\|_{\unital{\A}} \leq\varepsilon \text{.}
\end{equation*}

Thus $h\mulip(\A,\Lip,\mu)h $ is totally bounded for the topology of the norm $\|\cdot\|_{\unital{\A}}$.

Our construction of $h$ depends on the choice of the local state $\mu$: namely, we have thus far shown that for all local state $\mu$ of $(\A,\M)$, there exists a strictly positive element $h$ with the desired property. We now prove that in fact, the element $h$ constructed above for a given local state $\mu$ satisfies that $h \mulip(\A,\Lip,\nu) h$ is totally bounded in norm for any local state $\nu$ of $(\A,\M)$. 

Let $\nu \in \StateSpace(\A|\M)$. Let $(a_n)_{n\in\N}$ be a sequence in $\mulip(\A,\Lip,\nu)$. By definition, $(a_n - \mu(a_n) \unit_{\unital{\A}})_{n\in\N}$ is a sequence in $\mulip(\A,\Lip,\mu)$. Since $h\mulip(\A,\Lip,\mu)h$ is totally bounded for $\|\cdot\|_{\unital{\A}}$ and $\unital{\A}$ is a complete metric space for the distance induced by $\|\cdot\|_{\unital{\A}}$, there exists a subsequence $(a_{\gamma(n)})_{n\in\N}$ of $(a_n)_{n\in\N}$ such that:
\begin{equation*}
\left(h \left( a_{\gamma(n)} - \mu\left(a_{\gamma(n)}\right) \right) h\right)_{n\in\N}
\end{equation*}
converges in norm. Now, since $\nu,\mu$ are local states, we have $\Kantorovich{\Lip}(\mu,\nu) < \infty$. Thus, for all $n\in\N$, since we have $\nu(a_n)=0$, we conclude, by definition of $\Kantorovich{\Lip}$, that $\mu(a_n) \leq \Kantorovich{\Lip}(\mu,\nu)$. Thus $(\mu(a_{\gamma(n)})_{n\in\N}$ is a bounded sequence in $\R$. Consequently, there exists a convergent subsequence $(\mu(a_{\gamma\circ\delta(n)})_{n\in\N}$ of $(\mu(a_{\gamma(n)}))_{n\in\N}$. Hence, the sequence $(ha_{\gamma\circ\delta(n)}h)_{n\in\N}$ converges in norm in $\unital{\A}$. Therefore, $h\mulip(\A,\Lip,\nu)h$ has compact closure for $\|\cdot\|_{\unital{\A}}$, i.e. $h\mulip(\A,\Lip,\nu)h$ is totally bounded for $\|\cdot\|_{\unital{\A}}$.

The second assertion implies the third trivially.

Assume now that there exists a strictly positive $h \in \M$ and some local state $\mu\in\StateSpace(\A|\M)$ such that $h\mulip(\A,\Lip,\mu) h$ is totally bounded for $\|\cdot\|_{\unital{\A}}$. Let $a,b$ be two compactly supported elements in $\M$. Since $h$ is strictly positive in $\M$ and $a$ and $b$ are compactly supported, there exists $s,t \in \M$ such that $sh = a$ and $ht = b$. To be specific, and illustrate why it matters that $a$ (and $b$) are compactly supported, note that $h$ is bounded below on the support of $a$, since the latter is compact and $h$ is continuous. As $h$ is strictly positive, there exists $r>0$ such that $r \leq h$ on the support of $a$. Thus we can define a \emph{bounded} element $s$ by $s : x \in \sigma(\M)\mapsto \frac{a(x)}{h(x)} \leq \frac{\|a\|_{\unital{\A}}}{r} <\infty$ . Of course, $s$ is also compactly supported, with the same support as $a$. If $a$ was not compactly supported, the division may lead to an unbounded element. 

 We then have for all $c\in\mulip(\A,\Lip,\M)$:
\begin{equation}
\|acb\|_{\unital{\A}} = \|sh c ht\|_{\unital{\A}} \leq \|s\|_{\unital{\A}}\|t\|_{\unital{\A}} \|hch\|_{\unital{\A}} \text{.}
\end{equation}

Hence, since $h\mulip(\A,\Lip,\M)h$ is totally bounded in $\|\cdot\|_{\unital{\A}}$, so is $a\mulip(\A,\Lip,\M)b$. Thus by Theorem (\ref{qms-characterization-cc-thm}), the Lipschitz triple $(\A,\Lip,\M)$ is a {\qmss}.
\end{proof}

This completes our foundations for {\qms s} theory. We turn our attention to some examples.

\section{Examples}

We choose to illustrate our results with some fundamental examples. We start by casting the notion of a locally compact metric space in the framework of {\qms s}, of course. We then show that Rieffel's notion of compact quantum metric spaces and our previous notion of separable bounded quantum locally compact metric spaces both are special cases of our approach. We then give a first, simple example of metric on the algebra of compact operators of a separable Hilbert space, which shows some of the noncommutative phenomena which led us to our approach. 

We then conclude with the main example of this section, which is the Moyal plane with a natural spectral triple introduced in \cite{Varilly04}. This shows a very natural example of an unbounded {\qms} on a simple C*-algebra.

\subsection{Locally compact metric spaces}
Our first example shows how the concept of a locally compact metric space fits within the framework of {\qms}.

\begin{theorem}\label{classical-lcms-thm}
Let $(X,m)$ be a locally compact metric space and let $C_0(X)$ be the C*-algebra of complex valued continuous functions on $X$ vanishing at infinity. Let $\mathsf{Lip}_m$ be the Lipschitz seminorm:
\begin{equation*}
\mathsf{Lip}_m : f \in C_0 \longmapsto \sup \left\{ \frac{|f(y)-f(x)|}{m(y,x)} : x,y \in X \text{ and } x\not= y\right\} \text{.}
\end{equation*}
The Lipschitz triple $(C_0(X),\mathsf{Lip}_m,C_0(X))$ is a {\qms}. 
\end{theorem}

\begin{proof}
Let $x_0 \in X$ be any point. The Dirac probability measure $\delta_{x_0}$ is supported on the compact $\{x_0\}$, so it is a local state. Let:
\begin{equation*}
\mathfrak{L} = \mulip(C_0(X),\mathsf{Lip}_m,\delta_{x_0}) = \{ f \in \unital{C_0(X)} : \mathsf{Lip}_m(f)\leq 1 \text{ and } f(x_0) = 0 \}\text{.}
\end{equation*}
Let $K \in\compacts{X}$. Since $K$ is compact and $x \in X \mapsto m(x_0,x)$ is $1$-Lipschitz, hence continuous, there exists $R > 0$ such that $x\in K \implies m(x,x_0)\leq R$. Let $f \in \mulip(C_0(X),\mathsf{Lip}_m,\delta_{x_0})$. Then since $f$ is $1$-Lipschitz, we have $|f(x)|\leq|f(x)-f(x_0)|\leq R$ for all $x\in K$. Hence the set $\{ \indicator{K}f : f \in \mathfrak{L} \}$ is bounded in norm and equicontinuous on $K$ (since all its members can be seen as $1$-Lipschitz functions for the restriction of $m$ to $K$). Consequently, by Arz{\'e}la-Ascoli, the set $\indicator{K}\mathfrak{L}$ is norm precompact (for the norm of $L^\infty(X)$). By Theorem (\ref{qms-characterization-via-lu-topology-thm}) and since $\indicator{K}\mathfrak{L}\indicator{K} = \indicator{K}\mathfrak{L}$ as $C_0(X)$ is Abelian and $\indicator{K}$ is a projection, we conclude that $(C_0(X),\mathsf{Lip}_m,C_0(X))$ is a {\qms}.
\end{proof}

We note that, even for classical separable locally compact metric spaces, we can not use an arbitrary strictly positive element in Theorem (\ref{qms-characterization-separable-thm}). Indeed, if we equip $\R$ with its usual metric and denote the associated Lipschitz seminorm by $\mathsf{Lip}$, then we note that if $h : x\in \R \mapsto \frac{1}{\sqrt[4]{|x|+1}}$, then $h \in C_0(\R)$ is a strictly positive element, yet $h\mulip(C_0(\R),\mathsf{Lip},\delta_0)h$ is not totally bounded in norm --- it is in fact not even bounded. Yet, Theorem (\ref{classical-lcms-thm}) shows that $(C_0(\R),\mathsf{Lip},C_0(\R))$ is a {\qms}, which is obviously separable, and the Dirac probability measure $\delta_0$ at $0$ is obviously local with respect to this topography. Intuitively, strictly positive elements chosen in Theorem (\ref{qms-characterization-separable-thm}) must decay faster than the distance grows at infinity. This phenomenon will be illustrated again in the section below on bounded separable {\qms}, where indeed any strictly positive element could be used.

\subsection{Compact Quantum Metric Spaces}

To another extreme from the Abelian locally compact space, we find the class of compact quantum metric spaces introduced by Rieffel. We recall from \cite{Rieffel99} a characterization of these spaces. A pair $(\A,\Lip)$ is a compact quantum metric space if $\A$ is an order unit space and $\Lip$ is a seminorm defined on a dense subset of $\A$ containing the unit of $\A$ and such that:
\begin{enumerate}
\item $\Lip(a) = 0$ if and only if $a$ is a scalar multiple of the unit of $\A$,
\item The diameter of $\StateSpace(\A)$ is bounded for $\Kantorovich{\Lip}$,
\item The image of the set $\{ a\in \sa{\A} : \Lip(a)\leq 1 \}$ by the linear quotient surjection $\pi : \A \mapsto \bigslant{\A}{\C\unit_\A}$ is norm precompact for the quotient norm on $\bigslant{\A}{\C\unit_\A}$.
\end{enumerate}

These conditions are equivalent to $\Kantorovich{\Lip}$ metrizing the weak* topology on $\StateSpace(\A)$. Condition (3) can be replaced with the requirement that $\{a \in \sa{\A} : \Lip(a)\leq 1\text{ and } \|a\|\leq r\}$ is norm precompact when $r$ is the diameter of $(\StateSpace(\A,\Kantorovich{\Lip}))$.
Of course, our framework for {\qms s} requires us to work with C*-algebras rather than order-unit spaces. Within this limitation, we have:

\begin{theorem}
Let $(\A,\Lip)$ be a Lipschitz pair with $\A$ unital. An Abelian C*-subalgebra $\M$ of $\A$ is a topography for $\A$ if and only if $\M$ is unital with the same unit as $\A$, and the following are equivalent:
\begin{enumerate}
\item $(\A,\Lip)$ is a compact quantum metric space,
\item $(\A,\Lip,\C \unit_{\unital{\A}})$ is a {\qms},
\item For all Abelian C*-subalgebras $\M$ of $\A$ with $\unit_{\unital{\A}} \in \M$, the Lipschitz triple $(\A,\Lip,\M)$ is a {\qmss},
\item There exists an Abelian C*-subalgebra $\M$ of $\A$ with $\unit_{\unital{\A}} \in \M$ such that the Lipschitz triple $(\A,\Lip,\M)$ is a {\qms},
\item For all strictly positive elements $h\in \A$, the Lipschitz triple $(\A,\Lip,C^\ast(h))$ is a {\qmss},
\item There exists a strictly positive element $h\in \A$ such that the Lipschitz triple $(\A,\Lip,C^\ast(h))$ is a {\qms}.
\end{enumerate}
\end{theorem}

\begin{proof}
If $\M$ is an Abelian C*-subalgebra of $\A$ such that $(\A,\M)$ is a topographic quantum space, then $\M$ contains an approximate unit for $\A$, and since $\A$ is unital, this approximate unit converges in norm to $\unit_{\unital{\A}}$. Since $\M$ is closed, $\M$ is unital with unit $\unit_{\unital{\A}}$. Conversely, if $\M$ is an Abelian C*-subalgebra of $\A$ with $\unit_{\unital{\A}} \in \M$ then by definition, $(\A,\M)$ is a topographic quantum space.

We also note that if $\A$ is unital and $h \in \A$ is strictly positive in $\A$, then $C^\ast(h)$ contains an approximate unit, so $\unit_{\unital{\A}} \in C^\ast(h)$ (in fact, in this case $h$ is invertible). Thus for all strictly positive $h \in \A$, the pair $(\A,C^\ast(h))$ is a topographic quantum space.

Assume now that $(\A,\Lip)$ is a compact quantum metric space. Note that in particular, $(\StateSpace(\A),\Kantorovich{\Lip})$ is a compact metric space, so it is separable. Since any continuous linear functional  on $\A$ is the linear combination of four states, we conclude $\A^\ast$ is separable, and thus $\A$ is separable.

Let $\mu$ be a state.The state space $\StateSpace(\A)$ has finite radius for the {\mongekant}, i.e. for some $r\in\R$: 
\begin{equation*}
\sup\{ |\varphi(a)-\mu(a)| : a\in\sa{\A}\text{ and }\Lip(a) \leq 1, \varphi \in \StateSpace(\A) \} = r\text{.}
\end{equation*}
Equivalently, $\sup \{ \sup\{ |\varphi(a)| : \varphi \in \StateSpace(\A) \} : a\in\mulip(\A,\Lip,\mu) \} = r$, i.e.:
\begin{equation*}
\mulip(\A,\Lip,\mu) \subseteq \{ a \in \sa{\A} : \Lip(a)\leq 1\text{ and } \|a\|_\A \leq r\}
\end{equation*}
with the latter set being norm totally bounded since $(\A,\Lip)$ is a compact quantum metric space. Thus $(\A,\Lip,\C \unit_{\unital{\A}})$ is a {\qmss} by Theorem (\ref{qms-characterization-separable-thm}). This proves that (1) implies (2).

Assume now that $(\A,\Lip,\C\unit_{\unital{\A}})$ is a {\qmss} and let $\M$ be any Abelian C*-subalgebra of $\A$ with $\unit_{\unital{\A}} \in \M$. By definition, $\M$ is a topography of $\A$. Let $c,d \in \M$ be compactly supported positive elements in $\M$. Let $\mu \in \StateSpace(\A) = \StateSpace(\A|\M)$. The map $\theta : a \in \A \mapsto c a d$ is linear and bounded with norm bounded above by $\|c\|\|d\|$, so the image of the norm-totally bounded set $\mulip(\A,\Lip,\mu)$ by $\theta$ is a totally bounded set in norm, and is equal to $c \mulip(\A,\Lip,\mu)d$. This proves (2) implies (3) by Theorem (\ref{qms-characterization-cc-thm}). Of course, (3) implies (4), (5) and (6).

Assume that for some strictly positive element $h$, the set $(\A,\Lip,C^\ast(h))$ is a {\qmss}. Then $\unit_{\unital{\A}}\in C^\ast(h)$ is compactly supported, thus any state $\mu$ of $\A$ is local in $(\A,C^\ast(h))$ and by Theorem (\ref{qms-characterization-cc-thm}), the set $\mulip(\A,\Lip,\mu)$ is totally bounded in norm. Thus (6) implies (2).

Last, assume that $(\A,\Lip,\C\unit_{\unital{\A}})$ is a {\qmss}. Therefore, $(\A,\Lip,\C\unit_{\unital{\A}})$ is regular, and since $\unit_{\unital{\A}}$ is a projection in $\C\unit_{\unital{\A}}$, the state space $\StateSpace(\A) = \{\varphi \in \StateSpace(\A) : \varphi(\unit_{\unital{\A}})=1\}$ has finite diameter for $\Kantorovich{\Lip}$. Thus, $\mulip(\A,\Lip,\mu)$ is norm bounded for any state $\mu$ by Theorem (\ref{qms-characterization-separable-thm}). Now, the image of $\mulip(\A,\Lip,\mu)$ by the quotient map $\pi: \A \mapsto \bigslant{\A}{\C\unit_{\unital{\A}}}$ of Banach spaces is thus norm precompact for the quotient norm on $\bigslant{\A}{\C\unit_{\unital{\A}}}$ and is easily checked to be the same as the image of $\{a\in \sa{\A} : \Lip(a)\leq 1\}$ by $\pi$. Thus by \cite{Rieffel99}, $(\A,\Lip)$ is a compact quantum metric space. Thus (2) implies (1) as desired.
\end{proof}

Thankfully, the reader can feel reassured that a compact {\qms} is indeed a compact quantum metric space and vice-versa, with no ambiguity in the terminology we introduced.

\subsection{Bounded Separable Quantum Metric Spaces}

Our work in \cite{Latremoliere05b} suggests that a \emph{bounded, non-compact} quantum locally compact separable metric space $(\A,\Lip)$ should be a Lipschitz pair such that:
\begin{condition}[Boundedness Condition]\label{boundedness-cond}
$\mathfrak{L} = \{a \in \sa{\A} : \Lip(a)\leq 1\}$ is norm bounded,
\end{condition}
\begin{condition}[Local Boundedness Condition]\label{total-boundedness-cond}
There exists a strictly positive element $h\in\A$ such that $h\mathfrak{L}h$ is totally bounded in norm.
\end{condition}
This definition for a bounded separable metric quantum space was used, for instance, in \cite{Bellissard10}. Note that unlike our new setup, the second condition is equivalent to asking for \emph{all} strictly positive elements $h \in \A$, the set $h\mathfrak{L}h$ is totally bounded in norm. We only give a name to these conditions to clarify our exposition in this section, since a priori we have now two notions of bounded quantum metric spaces.

This tentative definition has a few problems, and in fact, our attention in \cite{Latremoliere05b} was less on this notion and more on the notion of Bounded-Lipschitz distance, which metrizes the weak* topology on the whole state space even in the nonunital (and not necessarily bounded) case, as we discussed in the introduction. A drawback of this tentative definition is that it does not fit the compact case or, of course, the unbounded case at all (even for a compact quantum metric space, $\mathfrak{L}$ would not be norm bounded). 

Yet, Proposition (\ref{bounded-wu-eq-tu-prop}) shows that on bounded sets, the topographic uniform topology and the weakly uniform topology agree. Thus, we are able to show that bounded, noncompact {\qms s} are indeed given by this older approach of ours. We start with a first observation, which explains that in the bounded case, the notion of a tame set is redundant:

\begin{proposition}
Let $(\A,\Lip,\M)$ be a {\qms} where $\mathfrak{L} = \{ a \in \sa{\A} : \Lip(a)\leq 1\}$ is norm bounded. Then every weak* compact subset of $\StateSpace(\A)$ is tame.
\end{proposition}

\begin{proof}
Let $\mathscr{K}$  be a weak* compact subset of $\StateSpace(\A)$. By Theorem (\ref{tight-eq-precompact-thm}), the set $\mathscr{K}$ is tight. Now, for all $a\in \sa{\unital{\A}}$ with $\Lip(a)\leq 1$, for all $K\in\compacts{\sigma(\M)}$ and for all $\varphi \in \mathscr{K}$, we have by Cauchy-Schwarz:
\begin{equation}\label{bqms-eq0}
\begin{split}
|\varphi(a)-\varphi(\indicator{K}a\indicator{K})| &\leq |\varphi((\unit_{\unital{\A}}-\indicator{K})a\indicator{K})| +  |\varphi(\indicator{K}a(\unit_{\unital{\A}}-\indicator{K})|\\
&\quad +  |\varphi((\unit_{\unital{\A}}-\indicator{K})a(\unit_{\unital{\A}}-\indicator{K}))|\\
&\leq 3\sqrt{\varphi(\unit_{\unital{\A}} - \indicator{K})} \|a\|_{\unital{\A}} \text{.}
\end{split}
\end{equation}
Since $\mathfrak{L}_1$ is bounded in norm and $\mathscr{K}$ is tight, we conclude from Inequality (\ref{bqms-eq0}) that $\mathscr{K}$ is tame.
\end{proof}

We thus have the complete picture for {\qms} of bounded diameter:

\begin{theorem}
Let $(\A,\Lip)$ be a Lipschitz pair with $\A$ not unital. Then the following are equivalent:
\begin{enumerate}
\item $(\A,\Lip)$ satisfies Condition (\ref{boundedness-cond}) and Condition (\ref{total-boundedness-cond}),
\item There exists a strictly positive element $h$ such that $(\A,\Lip,C^\ast(h)
)$ is a {\qmss} whose state space has finite diameter for the {\mongekant} associated with $(\A,\Lip)$.
\end{enumerate}
\end{theorem}

\begin{proof}
Let $(\A,\Lip,\M)$ be a {\qmss} with finite diameter, say $r\in \R$. Let $\mu \in \StateSpace(\A|\M)$. Let $a\in \sa{\A}$ with $\Lip(a)\leq 1$. By assumption, for all $\varphi\in\StateSpace(\A)$ we have:
\begin{equation*}
|\mu(a) - \varphi(a)| \leq r \text{.}
\end{equation*}
Since $\A$ is not unital, $0 \in \sigma(a)$, and thus there exists a net $(\varphi_\alpha)_{\alpha \in I}$ such that $\lim_{\alpha \in I} \varphi_\alpha(a) = 0$. Therefore, $\sup \{ |\mu(a)| : a\in\sa{\A}\text{ and }\Lip(a)\leq 1\} \leq r$.

On the other hand, the set $\mulip(\A,\Lip,\M)$ is bounded in norm by Proposition (\ref{wasskant-alt-expression-prop}), since the diameter of $\StateSpace(\A)$ is an upper bound for $\sup \{ |\varphi(a)| : \varphi \in \StateSpace(\A), a \in \mulip(\A,\Lip,\mu) \}$, which is of course $\sup \{\|a\|_{\unital{\A}} : a \in \mulip(\A,\Lip,\mu)\}$. 

Now, since $(\A,\Lip,\M)$ is a {\qmss}, by Theorem (\ref{qms-characterization-separable-thm}), there exists $h \in \M$, strictly positive, such that $h\mulip(\A,\Lip,\mu)h$ is norm precompact. Let $(a_n)_{n\in\N}$ be a sequence in $\mathfrak{L}_1$. Then $(h(a_n-\mu(a_n)\unit_{\unital{\A}})h)_{n\in\N}$ admits a norm convergent subsequence $(h(a_{m(n)}-\mu(a_{m(n)})\unit_{\unital{\A}})h)_{n\in\N}$, as a sequence in the norm precompact $h\mulip(\A,\Lip,\mu)h$. In turn, the sequence $(\mu(a_{m(n)})_{n\in\N}$ is bounded in $\R$ so it admits a convergent subsequence $(\mu(a_{m\circ k(n)})_{n\in\N}$. It is then immediate that $(ha_{m\circ k(n)}h)_{n\in\N}$ converges in norm. Hence $h\mathfrak{L} h$ is precompact in norm. 

Hence $(\A,\Lip)$ satisfies Conditions (\ref{boundedness-cond}) and (\ref{total-boundedness-cond}).

Conversely, assume that $(\A,\Lip)$ meets Conditions (\ref{boundedness-cond}) and (\ref{total-boundedness-cond}). Then for any strictly positive element $h \in \A$, the set $h\{ a \in \sa{\A} : \Lip(a)\leq 1\} h$ is precompact in norm, and $\{ a \in \sa{\A} : \Lip(a)\leq 1\}$ is bounded in norm. Thus, for any state $\mu$ we have $\sup \{ |\mu(a)| : a \in \sa{\A}\text{ and }\Lip(a)\leq 1\} < \infty$. By proceeding as in the first half of this proof, we conclude that $h \mulip(\A,\Lip,\mu) h$ is precompact in norm in $\unital{\A}$. Thus $(\A,\Lip,C^\ast(h))$ is a {\qmss} whose state space has finite diameter by construction.
\end{proof}

The reader may thus forget the so-called boundedness and local total boundedness conditions, as the concept of {\qms} offers a more coherent and general context in which bounded {\qmss s} fit unambiguously.

We now turn to two classes of examples which give unbounded state spaces in the {\mongekant} and are noncommutative, even simple in the next section.

\subsection{A First Quantum Metric on the Compact Operators C*-algebra}

This section presents a {\qms} structure over the simple C*-algebra of compact operators. This serves two purposes. First, this is an example of a simple C*-algebra to which our theory applies. In general, such constructions are left for our coming publications, as they are often involved and would considerably extend this paper. Thus, we take advantage of this simpler construction as an illustration of our theory in an accessible yet nontrivial case. Second, this example illustrates a very important point: even if one restricts attention to states which are supported by a compact projection within a {\qms}, it is still possible to find states at infinite distance. Thus, our definition for regularity of a Lipschitz triple is really as general as one can hope.

\begin{proposition}\label{compact-metric-1-prop}
Let $\mathfrak{K}$ be the C*-algebra of compact operators on a separable Hilbert space $\mathcal{H}$. Let $(\zeta_n)_{n\in\N}$ be a Hilbert basis for $\mathcal{H}$. Let $P_n$ be the orthogonal projection of $\mathcal{H}$ onto $\mathrm{span}\{\zeta_0,\ldots,\zeta_n\}$ for all $n\in \N$. Let $(\omega_n)_{n\in\N}$ be a sequence of real numbers and $\alpha > 0$ such that for all $n\in\N$ we have $\omega_n \geq \alpha$. Define, for any $n\in\N$:
\begin{equation}
\Lip : T \in \mathfrak{K} \longmapsto  \sup \{ \omega_n^{-1} \| P_n a P_n \|_{\mathfrak{K}} : n \in \N \} \text{.}
\end{equation}
Let $\mathfrak{D}$ be the C*-subalgebra of $\mathfrak{K}$ consisting of compact operators which are diagonal in $(\zeta_n)_{n\in\N}$, i.e. the C*-algebra generated by $(P_n)_{n\in\N}$. Note that $\mathfrak{D}$ is Abelian.

Then $(\mathfrak{K},\Lip,\mathfrak{D})$ is a {\qms}.
\end{proposition}

\begin{proof}
Since $\omega_n$ is strictly positive for all $n\in\N$, we have $\Lip(a) = 0$ if and only if $a = 0$. Moreover, since $\omega_n^{-1} \leq \alpha^{-1}$ for all $n\in\N$, we have for all $a\in\mathfrak{K}$ that $\Lip(a) \leq \alpha^{-1} \|a\|_{\mathfrak{K}}$. Thus, the pair $(\mathfrak{K},\Lip)$ is a Lipschitz pair. Moreover, $(P_n)_{n\in\N}$ is an approximate unit for $\mathscr{K}$ in $\mathfrak{D}$, so $(\mathfrak{K},\Lip,\mathfrak{D})$ is a Lipschitz triple.

Let $\mu$ be the state of $\mathfrak{K}$ defined by $P_0 T P_0 = \mu(T)P_0$ for all $T\in\mathfrak{K}$. Since $\mu(P_0)=1$ and $P_0 \in \mathfrak{D}$, the state $\mu$ is local. Let $a\in \mulip(\mathfrak{K},\Lip,\mu)$. By definition, $\|P_n a P_n\|_{\mathfrak{K}} \leq \omega_n$ for all $n\in\N$. Thus $P_n\mulip(\mathfrak{K},\Lip,\mu)P_n$ is norm bounded in the finite dimensional matrix algebra $P_n\mathfrak{K}P_n$, so it is norm precompact. Now, if $Q \in \mathscr{D}$ is a projection, then $Q\mulip(\A,\Lip,\mu)Q \subseteq P_n\mulip(\A,\Lip,\mu)P_n$ for some $n\in\N$, so it is also norm precompact. By Theorem (\ref{qms-characterization-via-lu-topology-thm}), the Lipschitz triple $(\mathfrak{K},\Lip,\mathfrak{D})$ is a {\qms}.
\end{proof}

Now, keeping the notation of Proposition (\ref{compact-metric-1-prop}), let $\xi = \sum_{n\in \N} \frac{1}{n+1} \zeta_n$. An easy computation shows that, if $P_\xi$ is the orthogonal projection on $\C\xi$, the state $\varphi$, defined by $P_\xi T P_\xi = \varphi(T)P_\xi$ for all $T\in\mathfrak{K}$, is at infinite distance from the local state $\mu$, and thus $\{\varphi\}$ is not tame. However, $\varphi(P_\xi) = 1$ and $P_\xi$ is a compactly supported projection \cite{Akemann89}. Thus we see that requiring all weak* compact convex faces of $\StateSpace(\mathfrak{K})$ to be tame would be too strong and exclude this metric (and the one in the next section). This is quite different from the commutative case, where all weak* compact convex faces are tame.

Moreover, let $\mathfrak{G}$ be the C*-subalgebra of $\mathfrak{K}$ of operators diagonal in some Hilbert base containing the vector $\|\xi\|_{\mathcal{H}}^{-1}\xi$ with $\xi\in\mathcal{H}$ defined above. Then we see that $(\mathfrak{K},\Lip,\mathfrak{G})$ is \emph{not} a {\qms}, as it is in fact not even a regular Lipschitz triple. This illustrates the dependence of our concept on the choice of a topography, in general. 

\subsection{A metric on the Moyal plane}

We now present a {\qmss} structure on the C*-algebra of compact operators on a separable Hilbert space, seen as the Moyal plane. We shall heavily rely on the computations found in \cite{Cagnache11} for this section, and we refer to \cite{GraciaBondia88a, GraciaBondia88b, Varilly04} for detailed expositions on the Moyal plane as a noncommutative geometric object. The main result of this section is that our framework applies to this very important example. We only present the material we need to establish our result, as any reasonable presentation of the Moyal plane would go beyond the scope of this paper. 

Fix $\theta > 0$. The Moyal plane $\mathfrak{M}_\theta$ is informally the quantum phase space of the quantum harmonic oscillator. It is a strict quantization of the usual plane $\R^2$ toward the canonical Poisson bracket on $C_0(\R^2)$, rescaled by a ``Plank constant'' $\theta$.  There are many *-algebras describing observables on the Moyal plane, associated with various degrees of differentiability. However, our result is concerned with the C*-algebra of continuous observables on the Moyal plane, which is the C*-algebra $\mathfrak{M}_\theta = C^\ast (\R^2,\sigma_\theta)$ where:
\begin{equation*}
\sigma : (p_1,q_1),(p_2,q_2) \in \R^2\times \R^2 \longmapsto \exp(2i\pi \theta (p_1q_2 - p_2q_1))
\end{equation*}
is a bicharacter on $\R^2$. This C*-algebra is easily seen to be *-isomorphic to the C*-algebra $\mathfrak{K}$ of compact operators on $L^2(\R)$. However, we follow here the standard presentation of the Moyal plane, which uses a twisted product (rather than a twisted convolution) obtained by conjugating the twisted convolution by the Fourier transform.

We now turn to the technical elements from  \cite{GraciaBondia88a, GraciaBondia88b, Varilly04} needed for our result. Fix $\theta>0$. Let $\mathcal{S}$ be the space of $\C$-valued Schwartz functions on $\R^2$. For any $f,g\in \mathcal{S}$ we define:
\begin{equation}
f\star g : x\in \R^2 \mapsto \frac{1}{\left(\pi\theta\right)^2}\iint_{\R^2\times \R^2} f(x+y)g(x+z)\sigma(y,z)\overline{\sigma(y,x)} \, dy dz \text{.}
\end{equation}
The pair $(\mathcal{S},\star)$ is an associative *-algebra, and is a *-algebra  which we denote by $\mathcal{S}_\theta$ if one takes complex conjugation as the *-operation. The integral defines a trace on $\mathcal{S}_\theta$.

We write $x$ for the function $x : (t,u) \in \R^2 \mapsto t$ and $p$ for $p : (t,u)\in \R^2 \mapsto u$. Denote $z = \frac{1}{\sqrt{2}}(x+ip)$. Let $f_{0,0} = \frac{1}{\sqrt{2\pi\theta}}\exp\left(2\pi \frac{x^2+p^2}{\theta} \right)$ be the Gaussian density of expectation zero and variance $\theta$ normalized so that $\|f_{0,0}\|_{L^2(\R^2)} = 1$. For any $n,m\in\N$ we set $f_{n,m} = \frac{1}{\sqrt{\theta^{m+n}m!n!}}\overline{z}^{\star n} \star f_{0,0} \star z^{\star m}$. We note that all elements of $\mathcal{S}$ are elements of $L^2(\R^2)$, and we denote the standard inner product on $L^2(\R^2)$ by $\left<\cdot,\cdot\right>$.
We have the following essential observations, denoting by $\delta_n^m$ the Kronecker symbol for $n,m$:
\begin{enumerate}
\item The family $(f_{{n,n}})_{n\in\N}$ is an orthonormal basis for $L^2(\R^2)$,
\item For all $n,m,p,q$ we have $\left<f_{{n,m}},f_{p,q}\right> = \delta_n^p\delta_m^q$,
\item For all $n,m,p,q \in \N$ we have $f_{{n,m}} \star f_{p,q} = \delta_m^p f_{n,q}$.
\end{enumerate}
As a result, we have an important matrix representation of $(\mathcal{S},\star)$ \cite{GraciaBondia88a,Cagnache11}. Let $\mathcal{M}_\theta$ be the algebra of so-called rapid-decay infinite matrices, i.e. the set of doubly indexed sequences $(a_{{m,n}})_{m, n\in\N}$ of complex numbers such that for all $k\in \N$ the series $\sum_{m,n \in \N} \theta^{2k}(m+1)^{k}(n+1)^{k}|a_{{m,n}}|^2$ converges. Thanks to this regularity condition, one may define the product $(a_{{m,n}})_{m,n\in\N}*(b_{{m,n}})_{m,n\in\N} = \left(\sum_{j \in \N} a_{nj}b_{jm}\right)_{m,n\in\N}$ for all $(a_{{m,n}})_{m,n\in\N}, (b_{{m,n}})_{m,n\in\N} \in \mathcal{M}_\theta$. With the obvious definition for adjoint, $(\mathcal{M}_\theta,*)$ is a *-associative algebra. Moreover, the map:
\begin{equation}\label{upsilon-eq0}
\Upsilon_\theta : c \in \mathcal{S} \longmapsto \left(\left<c,f_{{m,n}}\right>\right)_{m,n\in\N} \in \mathcal{M}_\theta
\end{equation}
is an isomorphism of *-algebras. A first consequence of this isomorphism is that since $\mathcal{M}_\theta$ acts naturally on $\ell^2(\N)$ as Hilbert-Schmidt operators, hence compact operators, $\Upsilon$ defines essentially a representation of $\mathcal{S}_\theta$ on $\ell^2(\N)$ by compact operators. In particular, $\mathcal{M}_\theta$ inherits a C*-norm $\|\cdot\|_{\mathcal{M}_\theta}$, namely the operator norm for its action on $\ell^2(\N)$.

One may then define a C*-norm on $\mathcal{S}_\theta$ so that $\Upsilon_\theta$ is an isometry from this norm to $\|\cdot\|_{\mathcal{M}_\theta}$. As a result, the C*-completion $\mathfrak{M}_\theta$ of $\mathcal{S}_\theta$ is *-isomorphic to the C*-completion of $(\mathcal{M}_\theta,*)$, i.e. the compact operators.

On the other hand, there exists a natural spectral triple on $\mathfrak{M}_\theta$. First, let $\pi$ be the representation $f \in \mathcal{S}_\theta \mapsto [g \in L^2(\R^2) \mapsto f\star g]$ --- one checks this is a well-defined *-representation and can be extended to $\mathfrak{M}_\theta$. For any nonzero vector $u \in \R^2$, we write $\frac{\partial}{\partial u}$ for the directional derivative along $u$, seen as as unbounded operator of  $L^(\R^2)$. Denote by $\partial$ the partial derivative $\frac{\partial}{\partial \left(\frac{\sqrt{2}}{2},\frac{\sqrt{2}}{2}\right)} = \frac{\sqrt{2}}{2}\left(\frac{\partial}{\partial (1,0)} - i\frac{\partial}{\partial (0,1)}\right)$ on $L^2(\R^2)$. Then we define the following operators on $L^2(\R^2)\otimes \C^2$:
\begin{align*}
\forall c \in \mathfrak{M}_\theta \;\; \Pi(c)&= \begin{pmatrix}
\pi(c) & 0 \\ 0 & \pi(c) \end{pmatrix} & 
D &= -i\sqrt{2}\begin{pmatrix} 0 & \overline{\partial} \\ \partial & 0 \end{pmatrix} \text{.}
\end{align*}

Then by \cite{Varilly04} $(\mathcal{S}_\theta,\Pi,D)$ is a candidate for a spectral triple for the Moyal plane $\mathfrak{M}_\theta$. In particular, $\Pi$ is a *-representation of $\mathfrak{M}_\theta$ on $L^2(\R^2)\otimes \C^2$, and the set $\{ a \in \sa{\mathcal{S}_\theta} : \|[\Pi(a),D]\|_{\mathfrak{B}(L^2(\R^2\otimes\C^2))} < \infty \} = \sa{\mathcal{S}_\theta}$ is norm dense in $\mathfrak{M}_\theta$. Moreover, since $\Pi$ is faithful, one checks that for all $a\in \mathcal{S}_\theta$, if $[\Pi(a),D] = 0$ then $a = 0$ \cite{Varilly04}.

Our interest lies with the Lipschitz pair $(\mathfrak{M}_\theta,\Lip_\theta)$ where:
\begin{equation}\label{moyal-lip-eq}
\Lip_\theta : c \in \mathcal{S}_\theta \mapsto \|[D,\Pi(c)]\|_{\mathfrak{B}(L^2(\R^2)\otimes\C^2)}\text{.}
\end{equation}

We extract the following result from \cite{Cagnache11}, which contains the part of their computation which is enough for our conclusion.

\begin{theorem}[\cite{Cagnache11}]\label{Moyal-plane-all-done-thm}
Let $\theta>0$ and $a\in \sa{\mathcal{S}_\theta}$. Let:
\begin{equation*}
\begin{split}
(a_{{m,n}})_{m,n\in\N} &= \Upsilon_\theta(a) \text{,}\\
(\alpha_{{m,n}})_{m,n\in\N} &= \Upsilon_\theta(\partial a)\text{,}\\
(\beta_{{m,n}})_{m,n\in\N} &= \Upsilon_\theta(\overline{\partial} a)\text{,}
\end{split}
\end{equation*}
with $\Upsilon_\theta$ defined by Equation (\ref{upsilon-eq0}). Then:
\begin{enumerate}
\item \emph{(Proposition 3.3 in \cite{Cagnache11})} For all $m,n \in \N$ with $n+m>0$,  we have:
\begin{equation*}
a_{{n,m}} = a_{0,0}\delta_n^m + \sqrt{\theta} \sum_{k=0}^{\min\{m,n\}} \frac{\alpha_{n-k,m-k-1} + \beta_{n-k-1,m-k}}{\sqrt{n-k} + \sqrt{m-k}}
\end{equation*}
where $\delta$ is the Kronecker symbol,
\item \emph{(Lemma 3.4 in \cite{Cagnache11})} If $\Lip_\theta(a)\leq 1$ then for all $n,m\in \N$ we have:
\begin{equation*}
\max\{ |\alpha_{{n,m}}|,|\beta_{{n,m}}|\} \leq \frac{\sqrt{2}}{2}\text{.}
\end{equation*}
\end{enumerate}
\end{theorem}

We now prove that the spectral triple constructed in \cite{Varilly04, Cagnache11} provides a {\qms} structure to the Moyal plane. To this end, the obvious topography is given by choosing the maximal Abelian C*-subalgebra $\mathfrak{D}$ of $\mathfrak{M}_\theta$ generated by $\{f_{{n,n}} : n \in \N\}$, which is C*-algebra of compact diagonal operators for the base $(f_{{n,n}})_{n\in\N}$ of $L^2(\R^2)$. Note that $\mathfrak{D} \cong c_0(\Z)$. With this in mind, we have:

\begin{theorem}
Let $\theta > 0$ and $\mathfrak{M}_\theta$ be the $\theta$-Moyal plane. Then $(\mathfrak{M}_\theta,\Lip_\theta,\mathfrak{D})$ is a {\qmss}.
\end{theorem}

\begin{proof}
First, note that for all $a\in\sa{\mathcal{S}_\theta}$, if $\Lip_\theta(a) = 0$ then $a = 0$, and the domain of $\Lip_\theta$ is $\mathcal{S}_\theta$, which is dense in $\mathfrak{M}_\theta$. Moreover, $\mathfrak{D}$ contains the approximate unit $(\sum_{k=0}^n f_{{k,k}})_{n\in\N}$ for $\mathfrak{M}_\theta$, so $(\mathfrak{M}_\theta,\Lip_\theta,\mathfrak{D})$ is a Lipschitz triple.

Let $\mu$ be defined by $\mu(a) = \left<af_{0,0},f_{0,0}\right>$ for all $a\in \mathfrak{M}_\theta$. Note that $\mu$ is a local state, since $f_{0,0} \in \mathfrak{D}$ and $\mu(f_{0,0})=1$. 
Now, if $a \in \sa{\unital{\mathfrak{M}_\theta}}$ then $a = b + \lambda \unit_{\unital{\mathfrak{M}_\theta}}$ for some $\lambda \in \R$. By definition, $\Lip_\theta(a)\leq 1$ if and only if $b\in\mathcal{S}_\theta$ and $\Lip_\theta(b)\leq 1$. Let:
\begin{equation*}
\begin{split}
(a_{{n,m}})_{n,m\in\N} &= \Upsilon_\theta(a)\text{,}\\
(b_{{n,m}})_{n,m\in\N} &= \Upsilon_\theta(b)\text{,}\\
(\alpha_{{n,m}})_{n,m\in\N} &= \Upsilon(\partial a) = \Upsilon(\partial b)\text{,}\\
(\beta_{{n,m}})_{n,m\in\N} &= \Upsilon(\overline{\partial} a) = \Upsilon(\overline{\partial} b)\text{.}
\end{split}
\end{equation*}
By Theorem (\ref{Moyal-plane-all-done-thm}), we have:
\begin{equation*}
b_{{n,m}} = b_{0,0} \delta_n^m + \sqrt{\theta} \sum_{k=0}^{\min\{m,n\}} \frac{\alpha_{n-k,m-k-1} + \beta_{n-k-1,m-k}}{\sqrt{n-k} + \sqrt{m-k}}\text{.}
\end{equation*}
Hence, if $\mu(a) = 0$, then $\lambda = -b_{0,0}$ and we get:
\begin{equation*}
a = \sqrt{\theta} \sum_{k=0}^{\min\{m,n\}} \frac{\alpha_{n-k,m-k-1} + \beta_{n-k-1,m-k}}{\sqrt{n-k} + \sqrt{m-k}}\text{.}
\end{equation*}

By Theorem (\ref{Moyal-plane-all-done-thm}), we then see that if $a\in \sa{\unital{\mathfrak{M}_\theta}}$ with $\Lip_\theta(a)\leq 1$ and $\mu(a)= 0$ then, for all $n,m \in \N$ we have:
\begin{equation}\label{moyal-eq0}
|a_{{n,m}}| \leq 2\sqrt{2\theta}\left(\sum_{k=0}^{\min\{m,n\}} \frac{1}{\sqrt{m-k}+\sqrt{n-k}}\right)\text{.}
\end{equation}
Since $\sigma(\mathfrak{D}) = \Z$, a subset $K\subseteq \sigma(\mathfrak{D})$ is compact if and only if it is finite, and the projection $\indicator{K}$ is the finite rank projection $\sum_{n\in K} f_{{n,n}}$.We thus have, for all $a\in \mulip(\mathfrak{M}_\theta,\Lip_\theta,\mu)$:
\begin{equation*}
\begin{split}
\left\|\left(\sum_{n\in K} f_{{n,n}}\right) a \left(\sum_{n\in K} f_{{n,n}}\right)\right \|_{\mathfrak{M}_\theta} &\leq \sum_{n\in K}\sum_{m\in K}\|\Upsilon_\theta(f_{{n,n}}a f_{{m,m}})\|_{\mathfrak{B}(\ell^2(\N))} \\
&= \sum_{n\in K}\sum_{m\in K} |a_{{m,n}}|
\end{split}
\end{equation*}
which is uniformly bounded over $\mulip(\mathfrak{M}_\theta,\Lip_\theta,\mathfrak{D})$ by Inequation (\ref{moyal-eq0}). Since the set $\indicator{K}\mulip(\mathfrak{M}_{\theta},\Lip_\theta,\mathfrak{D})\indicator{K}$ is a bounded subset of the finite dimensional C*-algebra $\indicator{K} \mathfrak{M}_\theta \indicator{K}$ (of dimension $|K|^2$, to be precise), it is a precompact set. Thus by Theorem (\ref{qms-characterization-via-lu-topology-thm}), the Lipschitz triple $(\mathfrak{M}_\theta,\Lip_\theta,\mathfrak{D})$ is a {\qms}.
\end{proof}

The {\qms} $(\mathfrak{M}_\theta,\Lip_\theta,\mathscr{D})$ has infinite diameter, as shown for instance in Proposition (3.6) in \cite{Cagnache11}, where a sequence $(\omega_n)_{n\in\N}$ of pure states of the Moyal plane is proven to satisfy $\Kantorovich{\Lip_\theta}(\omega_0,\omega_n) = \sqrt{\frac{\theta}{2}}\sum_{k=1}^n \frac{1}{\sqrt{k}}$ for all $n\in\N$, so $(\omega_n)_{n\in\N}$ is an unbounded sequence in $(\StateSpace(\mathfrak{M}_\theta),\Kantorovich{\Lip_\theta})$. Thus, this example provides us with a solid, natural example for the theory of this paper. We leave the study of other examples for subsequent papers.


\bibliographystyle{amsplain}
\bibliography{../thesis}

\providecommand{\bysame}{\leavevmode\hbox to3em{\hrulefill}\thinspace}
\providecommand{\MR}{\relax\ifhmode\unskip\space\fi MR }
\providecommand{\MRhref}[2]{%
  \href{http://www.ams.org/mathscinet-getitem?mr=#1}{#2}
}
\providecommand{\href}[2]{#2}
\begin{thebibliography}{10}

\bibitem{Akemann89}
{C.} {A}kemann, {J}. {A}nderson, and {G.}~{K.} {P}edersen, \emph{Approaching
  infinity in {$C^\ast$-algebras}}, Journal of Operator Theory \textbf{21}
  (1989), 255--271.

\bibitem{Alfsen01}
{E}. {A}lfsen and {F}. {S}hultz, \emph{State spaces of operator algebras},
  Birkh{\"a}user, 2001.

\bibitem{Bellissard10}
{J}. {B}ellissard, {M}. {M}arcolli, and {K}. {R}eihani, \emph{Dynamical systems
  on spectral metric spacea}, Submitted (2010), 46 pages, ArXiv: 1008.4617.

\bibitem{Billingsley}
{P}. {B}illigsley, \emph{Probability and measures}, Willey and Sons.

\bibitem{Cagnache11}
{E}. {C}agnache, {F} {D'A}ndrea, {P}. {M}artinetti, and {J}.~{C}. {W}allet,
  \emph{The spectral distance on the moyal plane}, J. Geom. Phys. \textbf{61}
  (2011), 1881--1897, ArXiv: 0912.0906.

\bibitem{Connes89}
A.~{C}onnes, \emph{Compact metric spaces, fredholm modules and
  hyperfiniteness}, Ergodic Theory and Dynamical Systems \textbf{9} (1989),
  no.~2, 207--220.

\bibitem{Connes}
\bysame, \emph{Noncommutative geometry}, Academic Press, San Diego, 1994.

\bibitem{Connes97}
{A}. {C}onnes, {M}. {D}ouglas, and {A}. {S}chwarz, \emph{Noncommutative
  geometry and matrix theory: Compactification on tori}, JHEP \textbf{9802}
  (1998), hep-th/9711162.

\bibitem{Dobrushin70}
{R}.~{L}. {D}obrushin, \emph{Prescribing a system of random variables by
  conditional probabilities}, Theory of probability and its applications
  \textbf{15} (1970), no.~3, 459--486.

\bibitem{Doplicher95}
S.~{D}oplicher, K.~{F}redenhagen, and J.~E. {R}oberts, \emph{The quantum
  structure of spacetime at the plank scale and quantum fields}, Communication
  in Mathematical Physics \textbf{172} (1995), 187--220.

\bibitem{Dudley}
R.~M. Dudley, \emph{Real analysis and probability}, 2002 ed., Cambridge Studies
  in Advanced Mathematics, vol.~74, Cambridge University Press, 2002.

\bibitem{Folland}
{G}. {F}olland, \emph{Real analysis: Modern techniques and their applications},
  2nd ed., Wiley Interscience, 1999.

\bibitem{Varilly04}
{V}. {G}ayral, {J}.~{M}. {G}raci {B}ondia, {B}. {I}ochum, {T}. {S}ch{\"u}cker,
  and {J}.~{C}. {V}arilly, \emph{Moyal planes are spectral triples}, Comm.
  Math. Phys. \textbf{246} (2004), 569--623, ArXiv: hep-th/0307241.

\bibitem{GraciaBondia88a}
J.~M. {G}racia {B}ondia and J.~C. {V}arilly, \emph{Algebras of distributions
  suitable for phase-space quantum mechanics {I}}, J. Math. Phys. \textbf{29}
  (1988), no.~4, 869--879.

\bibitem{GraciaBondia88b}
\bysame, \emph{Algebras of distributions suitable for phase-space quantum
  mechanics {II}}, J. Math. Phys. \textbf{29} (1988), no.~4, 880--887.

\bibitem{Gromov}
M.~{G}romov, \emph{Metric structures for {R}iemannian and non-{R}iemannian
  spaces}, Progress in Mathematics, Birkh{\"a}user, 1999.

\bibitem{Kadison91}
{R}.~{V}. {K}adison and {J}.~{R}. {R}ingrose, \emph{Fundamentals of the
  operator algebras {III}}, AMS, 1991.

\bibitem{Kantorovich40}
{L}.~{V}. {K}antorovich, \emph{On one effective method of solving certain
  classes of extremal problems}, Dokl. Akad. Nauk. USSR \textbf{28} (1940),
  212--215.

\bibitem{Kantorovich58}
{L}.~{V}. {K}antorovich and {G}.~{Sh}. {R}ubinstein, \emph{On the space of
  completely additive functions}, Vestnik Leningrad Univ., Ser. Mat. Mekh. i
  Astron. \textbf{13} (1958), no.~7, 52--59, In Russian.

\bibitem{Kuperberg10}
{G}. {K}uperberg and {N}. {W}eaver, \emph{A {V}on {N}eumann algebra approach to
  quantum metrics}, Mem. Amer. Math. Soc. \textbf{215} (2012), no.~1010, 1--80,
  ArXiv: 1005.0353.

\bibitem{Latremoliere05}
{F}. {L}atr{\'e}moli{\`e}re, \emph{Approximation of the quantum tori by finite
  quantum tori for the quantum gromov-hausdorff distance}, Journal of Funct.
  Anal. \textbf{223} (2005), 365--395, math.OA/0310214.

\bibitem{Latremoliere05b}
\bysame, \emph{Bounded-lipschitz distances on the state space of a
  {C*}-algebra}, Tawainese Journal of Mathematics \textbf{11} (2007), no.~2,
  447--469, math.OA/0510340.

\bibitem{Latremoliere12c}
\bysame, \emph{A topographic gromov-hausdorff hypertopology for quantum proper
  metric spaces}, In preparation. (2012).

\bibitem{Martinetti11}
{P}. {M}artinetti and {L}. {T}omassini, \emph{Noncommutative geometry of the
  moyal plane: translation isometries, connes spectral distance between
  coherent states, pythagoras inequality},  (2011), 29 pages, ArXiv: 1110.6164.

\bibitem{Ozawa05}
{N}. {O}zawa and M.~A. {R}ieffel, \emph{Hyperbolic group {$C\sp\ast$}-algebras
  and free products {$C\sp\ast$}-algebras as compact quantum metric spaces},
  Canad. J. Math. \textbf{57} (2005), 1056--1079, ArXiv: math/0302310.

\bibitem{Pedersen79}
{G}.~{K}. {P}edersen, \emph{{C}*-{A}lgebras and their automorphism groups},
  Academic Press, 1979.

\bibitem{prohorov56}
{Y}.~V. {P}rohorov, \emph{Convergence of random processes and limit theorems in
  probability theory.}, Teor. Veroyatnost. i Primenen. (translated as Theory
  Probab. Appl.) \textbf{1} (1956), 177--238.

\bibitem{Rieffel96}
M.~A. {R}ieffel, \emph{On the operator algebra for the space-time uncertainty
  relations}, Operator algebras and quantum field theory (1996), 375--382.

\bibitem{Rieffel98a}
\bysame, \emph{Metrics on states from actions of compact groups}, Documenta
  Mathematica \textbf{3} (1998), 215--229, math.OA/9807084.

\bibitem{Rieffel99}
\bysame, \emph{Metrics on state spaces}, Documenta Math. \textbf{4} (1999),
  559--600, math.OA/9906151.

\bibitem{Rieffel02}
\bysame, \emph{Group {$C\sp\ast$}-algebras as compact quantum metric spaces},
  Documenta Mathematica \textbf{7} (2002), 605--651, ArXiv: math/0205195.

\bibitem{Rieffel01}
\bysame, \emph{Matrix algebras converge to the sphere for quantum
  {G}romov--{H}ausdorff distance}, Mem. Amer. Math. Soc. \textbf{168} (2004),
  no.~796, 67--91, math.OA/0108005.

\bibitem{Rieffel08}
\bysame, \emph{A global view of equivariant vector bundles and {D}irac
  operators on some compact homogenous spaces}, Contemporary Math \textbf{449}
  (2008), 399--415, ArXiv: math/0703496.

\bibitem{Rieffel10b}
\bysame, \emph{{L}eibniz seminorms for {``Matrix algebras converge to te
  sphere"}}, Clay Math. Proc. \textbf{11} (2010), 543--578.

\bibitem{Rieffel10}
\bysame, \emph{Vector bundles and gromov-hausdorff distance}, Journal of
  {K}-theory \textbf{5} (2010), 39--103, ArXiv: math/0608266.

\bibitem{Rieffel00}
\bysame, \emph{{G}romov-{H}ausdorff distance for quantum metric spaces}, Mem.
  Amer. Math. Soc. \textbf{168} (March 2004), no.~796, math.OA/0011063.

\bibitem{Stone48}
A.~H. {S}tone, \emph{Paracompactness and product spaces}, Bull. Amer. Math.
  Soc. \textbf{54} (1948), 977--982.

\bibitem{tHooft02}
G.~{T}'Hooft, \emph{Determinism beneath quantum mechanics}, Presentation at
  "Quo Vadis Quantum Mechanics?", Temple University, Philadelphia (2002),
  quant-ph/0212095.

\bibitem{Wasserstein69}
{L}.~{N}. {W}asserstein, \emph{Markov processes on a countable product space,
  describing large systems of automata}, Problemy Peredachi Infomatsii
  \textbf{5} (1969), no.~3, 64--73, In Russian.

\end{thebibliography}

\vfill

\end{document}